\documentclass[11pt]{article}

\usepackage{amssymb}
\usepackage{amsthm}
\usepackage{amsmath}
\usepackage{latexsym}
\usepackage{tikz}
\newtheorem{theorem}{Theorem}[section]
\newtheorem{proposition}[theorem]{Proposition}
\newtheorem{lemma}[theorem]{Lemma}

\newtheorem{remark}[theorem]{Remark}

\makeatletter \@addtoreset{equation}{section} \makeatother
\renewcommand{\theequation}{\thesection.\arabic{equation}}

\def\tilde{\widetilde}

\newcommand{\be}{\begin{equation}}
\newcommand{\ee}{\end{equation}}
\newcommand{\bes}{\begin{eqnarray}}
\newcommand{\ees}{\end{eqnarray}}

\def\ve{\varepsilon}

\def\bega{\begin{array}}
\def\enda{\end{array}}
\def\begi{\begin{itemize}}
\def\endi{\end{itemize}}
\def\bel{\begin{equation}\label}
\begin{document}

\title{ Existence and Uniqueness  of Global Weak  solutions of the  Camassa-Holm  Equation with a Forcing}
\author{ Shihui Zhu \footnote {Corresponding  E-mail: shihuizhumath@163.com; shihuizhumath@sicnu.edu.cn}   \\    Department of Mathematics,
Sichuan Normal University\\ Chengdu, Sichuan 610066, China }
 
\date{}
\maketitle

\begin{abstract}
In this paper, we study the global well-posedness for the Camassa-Holm(C-H) equation with a forcing  in $H^1(\mathbb{R})$ by the characteristic method. Due to the forcing, many important properties  to study the well-posedness of weak solutions do not   inherit  from the C-H equation without a forcing, such as conservation laws, integrability.
By exploiting  the balance law and some new estimates, we prove the existence and uniqueness of global weak solutions for the C-H equation with a forcing in $H^1(\mathbb{R})$.

\end{abstract}

2010\textit{\ Mathematical Subject Classification:}  35L05\quad 35D30\quad 76B15

 
\section{Introduction}

In this paper, we consider the C-H equation with  a forcing  in the following  form:
\be\label{E0} u_t-u_{txx}+3u u_x=2 u_xu_{xx}+u u_{xxx}+ku,\ee
where $u:=u(t,x): \mathbb{R}\times \mathbb{R}\rightarrow \mathbb{R}$. The differential operators are defined by $u_t:=\frac{\partial u}{\partial t} $, $u_x:=\frac{\partial u}{\partial x} $, $u_{xx}:=\frac{\partial^2 u}{\partial x^2} $, $u_{txx}:=\frac{\partial^3 u}{\partial t\partial x^2}$ and $u_{xxx}:=\frac{\partial^3 u}{\partial x^3} $.  $ku$ is the forcing term, where $k\in \mathbb{R}$ is a constant. 
 In particular, when $k=0$, Eq.(\ref{E0})  is a well-known integrable equation describing the velocity dynamics of shallow water waves, named Camassa-Holm equation, which models the propagation of unidirectional shallow water waves over a flat bottom by 
approximating the Green-Naghdi equations (see \cite{CL2009}), with formal 
Hamiltonian derivations provided in \cite{CH1993,FF1981}.

 In the last two decades, the C-H equation has attracted many attentions, and a lot of  interesting properties have been found, including integrability(\cite{CH1993,CGI2006,CM1999}), existence of peaked solitons and multi-peakons(\cite{ACHM1994,BSS1999}),   breaking waves(\cite{CE1998}). It should  point  out 
that tsunami modelling has been connected to the dynamics of the 
Camassa-Holm equation(see the discussion in \cite{CJ2008,Lakshmanan2007}).
   Among these important properties, the following three conservation laws are crucial in studying the C-H equation.
 \[F_1(u)=\int_{\mathbb{R}} (u-u_{xx})dx,\quad F_2(u)=\int_{\mathbb{R}} (u^2+u_x^2)dx,\quad F_3(u)=\int_{\mathbb{R}} (u^3+uu_x^2)dx.\]
This implies that  the C-H equation can be written in bi-Hamiltonian form which means a second, compatible structure:
 \[\partial_t F_2^{'}(u)=-\partial_xF_3^{'}(u),\]
 where $^{'}$ denotes the Fr\'echet derivative. Then,   the orbital stability of the peakon (a special global solution of the C-H equation in the form $u(t,x)=ce^{-|x-ct|}$, where $c\in \mathbb{R}$ is wave speed) is well understood  in \cite{CM2001,CS2000,DM2009}. With regards to the global existence of weak solutions for the Cauchy problem, in \cite{BC2007b,XZ2000,XZ2002},  the global existence and uniqueness of dispersive solutions is obtained as weak limits of viscous regularizations of the C-H equation in $H^1(\mathbb{R})$. The global existence of conservative solutions is studied in \cite{BC2007,HR2007} by the new characteristic method.
  Recently, the uniqueness of  conservative solutions  is proved in \cite{BCZ2015} by using the technique of a generalized  characteristic method.  
 
 However, for Eq.(\ref{E0}), due to the existence of the  forcing term $ku$,  the structure of Eq.(\ref{E0}) changes dramatically.  For example,  the peakon  $u(t,x)=ce^{-|x-ct|}$ is no longer a solution of Eq.(\ref{E0}), and in fact, from \cite{Ivanov2007}, Eq.(\ref{E0}) is not integrable. Most  important of all,  Eq.(\ref{E0}) does not have a Hamiltonian structure. A nature question is arising how to study the well-posedness of weak solutions  of Eq.(\ref{E0})? Whether there exists a global weak solution? And how about the uniqueness of the weak solution?

  Motivated by those questions, we try to apply  the new characteristic method established  in \cite{BC2007,BCZ2015,BCZ20152}, 
   to study the global well-posedness for the forcing C-H equation (\ref{E0}).  
We remark that  in \cite{BC2007}, the authors provided a new characteristic method to study the existence of global conservative solutions for the   C-H equation without forcing, relying heavily on the conservation of energy. But this is not essential.   The essential structure is the study of a balance law for the C-H equation.

Now, we briefly introduce our results and the key ideas of the proof.   
We first study the structure of Eq.(\ref{E0}). 
Let $p(x)=\frac 12 e^{-|x|}$, where $x\in \mathbb{R}$.  We collect some important properties of $p(x)$ in the following
 \begin{itemize}
\item [(i)] $\mathcal{F}[p] =\frac{1}{1+|\xi|^2}$.
\item [(ii)] $\mathcal{F}[p- p_{xx}]=1$ implies $p-p_{xx}=\delta(x)$, where $\delta$ is the Dirac function.
\item [(iii)] $(1-\partial_x^2)^{-1}u=p*u(x):=\int_{-\infty}^{+\infty}p(x-y)u(y)dy$.
\item [(iv)] $p*(u-u_{xx})=u$,
\end{itemize}
where $\mathcal{F}$  denotes the Fourier transformation in $\mathbb{R}$. Then,   Eq.(\ref{E0}) can be expressed in the following form.
\be\label{E}u_t+u u_x+P_x-kQ=0,
\ee where the singular intergral operators $P$ and $Q$ are defined by  \be\label{PQ}P: =  \int_{\mathbb{R}} p(x-y) [u^2+\frac 12 u_x^2](y)dy\ \ {\rm and}\ \ Q:=  \int _{\mathbb{R}}p(x-y)u(y)dy.\ee 
 For the smooth solutions of Eq.(\ref{E}), we differentiate (\ref{E}) with respect to $x$.  
\be\label{E1} u_{tx}+uu_{xx}+\frac{1}{2}  u_x^2-u^2+P-kQ_x=0.\ee
By multiplying $2u_x$, fortunately, we  find that Eq.(\ref{E}) satisfies the following balance law in the following form.
\be\label{Balance1} (u_x^2)_t+(uu_x^2)_x+2(-u^2+P-k Q_x)u_x=0. \ee
Indeed, (\ref{Balance1}) is called the balance law  due to the third term in (\ref{Balance1}) only includes  first-order derivative of $u$, although some nonlocal operators are involved. In the present paper, we   extend the arguments  for the C-H equation without a forcing 
 in \cite{BC2007,BCZ2015} to the forcing C-H equation (\ref{E})  by using only the balance law (\ref{Balance1}).  Actually, there are two main difficulties to study the existence and uniqueness of global weak solutions for the forcing  C-H equation (\ref{E}):  one is   
 the loss of conservation of energy, that is, 
\be\label{Difficulty} \int_{\mathbb{R}} [u^2(t,x)+u_x^2(t,x) ] dx \neq {\rm Constant};\ee
the other is the appearance of a new nonlocal singular operator $Q$.

We supplement Eq.(\ref{E}) with  the initial data
\be\label{EI}u(0,x)=u_0(x).\ee  Now, we state our main theorem for  the existence and uniqueness of global weak solutions for the Cauchy problem (\ref{E})-(\ref{EI}).
\begin{theorem}\label{main-theorem}
 Let $k\in \mathbb{R}$  be a constant. 
Suppose  $u_0\in H^1(\mathbb{R})$ is an absolutely continuous function on $x$. Then the Cauchy problem 
 (\ref{E})-(\ref{EI}) admits a global weak solution $u(t,x)\in H^1(\mathbb{R}) $ with 
 \be\label{weak-form}  \int_{\Gamma}\{-u_x \phi_t  -uu_x\phi_x+(-\frac12 u_x^2-u^2+P-kQ_x)\phi \}dxdt +\int_{\mathbb{R}} (u_0)_x \phi(0,x)dx=0 \ee
for every test function $\phi\in C_c^1(\Gamma)$ with $\Gamma= \{(t,x)\,\ |\  t\in \mathbb{R}, \ x\in \mathbb{R} \}$
Furthermore, the weak solution satisfies the following properties. 
\begin{itemize}
\item[(i)] $u(t,x)$ is $\frac12$-H\"older continuous with respect to $t$ and $x$, for t in any  bounded interval. 
\item[(ii)] For  every fixed $t\in \mathbb{R}$, the map $t\mapsto u(t,\cdot)$ is Lipschitz continuous under $L^2$-norm.
\item[(iii)]  The balance law (\ref{Balance1}) 
is satisfied in the following sense: there exists a family of Radon measures $\{\mu_{(t)},\, t\in \mathbb{R} \}$,  depending continuously on time and w.r.t the topology of weak convergence of measures, and   
for every $t\in \mathbb{R} $,  the absolutely continuous part of $\mu_{(t)}$ w.r.t. Lebesgue measure has density $u_x^2(t,\cdot)$, which provides a 
measure-valued solution to the balance law
\be\label{weak_en}  \int_{\Gamma} \{  u_x^2\phi_t+uu_x^2 \phi_x +  2u_x(u^2-P+k Q_x)  \phi\, \}dxdt    -\int_{\mathbb{R}} (u_{0})_{x}^{2}\phi(0,x)dx=0,  \ee
for every test function $\phi\in C_c^1(\Gamma)$.  
\item[(iv)]
Some continuous dependence result holds. Consider a sequence of initial data $ u_{0,n}$ such that 
$\|u_{0,n} -u_0\|_{H^1 }\rightarrow 0$, as $n\rightarrow+\infty$. Then
the corresponding solutions $u_n(t,x)$ converge to $u(t,x)$ uniformly for $(t,x)$
in any bounded sets.
\end{itemize}
\end{theorem} 
Finally, we  prove  the uniqueness of the global weak solutions of Eq.(\ref{E}) satisfying the balance law (\ref{weak_en}) in the last section of this paper, as follows. 
\begin{theorem}\label{unique}
 Let $k\in \mathbb{R}$  be a constant.  For any initial data $u_0\in H^1(\mathbb{R})$, the Cauchy problem (\ref{E})-(\ref{EI}) has a unique global weak   solution satisfying (\ref{weak_en}). 
\end{theorem}
We   point out that  in this paper we verified that the methods for both existence and uniqueness of  weak solutions
used  in \cite{BC2007,BCZ2015} can be extended to Eq.(\ref{E}) by exploring only a balance law (\ref{Balance1}) instead of a conservation of energy.  However, the jump for using a balance law instead of a conservation law  is very nontrivial.  This idea has also been used for a generalized wave equation with a higher-order nonlinearity in our forthcoming paper  \cite{CSZ2015}. 

%%%%%%%%%%%%%

\renewcommand{\theequation}{\thesection.\arabic{equation}}
\setcounter{equation}{0}
\section{The  Transferred System}

 In this paper, we firstly study the global  weak solutions of (\ref{E})-(\ref{EI}) in $H^1(\mathbb{R})$, in terms of Bressan and Constantin's arguments in \cite{BC2007}(see also\cite{CS2015}). First, we shall use the characteristic method to transfer the quasi-linear Eq.(\ref{E}) to a semi-linear system in terms of smoothing solutions. Then, we 
prove the existence of global  weak solutions for the transferred system by the ODE's argument. 
Finally, by the inverse transformation, we return to the original problem. 
    More precisely,  
 we introduce a new coordinate  $(T,Y)$ defined
\be\label{2.1}(t,x)\longrightarrow (T,Y),\ \ \ {\rm where }\ \ \ T=t, \ \ Y=\int_0^{x_0(Y)} (1+(u_0)_x^2) dx.\ee
 The equation of the characteristic is 
 \be\label{C}\frac{dx(t,Y)}{dt}=u (t,x(t,Y))\ \ \ \ {\rm with }\ \ \ x(0,Y)=x_0(Y).\ee
Here, $Y=Y(t,x)$ is  a characteristic coordinate and satisfies
$Y_t+u Y_x=0$
for any $(t,x)\in \mathbb{R}\times \mathbb{R} $.  After some computations, we see that for any smoothing function $f:=f(T,Y(t,x))$ have the following properties, under the new coordinates $(T,Y)$: 
 \[f_T=f_T(\ T_t+u T_x)+f_Y\ (Y_t+u   Y_x)=f_t+u f_x,\]
 \[f_x=f_T\ T_x+ f_Y\ Y_x=f_Y\ Y_x.\] 
Take the transformation:
 \be\label{T} v:=2\arctan u_x   \ \ \ {\rm and }\ \ \ \xi:=\frac{ 1+u_x^2 }{Y_x}.\ee
Under this transformation, we see that 
 \be\label{T1} u_x= \tan \frac v2\ \ \ \ \ \ \ \ \ 1+u_x^2=\sec^2\frac v2,\ \ \ \ \ \frac{1}{1+u_x^2}=\cos^2\frac v2,\ee
\be\label{T2}\frac{u_x^2}{1+u_x^2}=\sin^2\frac v2,\ \  \frac{u_x}{1+u_x^2}=\frac12\sin v,\ \ \ x_Y=\frac{\xi}{ 1+u_x^2  }=\cos^{2}\frac v2 \cdot \xi.\ee
Then, we will consider Eq.(\ref{E}) under the new characteristic coordinate $(T,Y)$. 
$u_T=u_t+u u_x=-P_x+kQ$. 
From (\ref{E1}), we deduce that 
\[\begin{array}{lll}
v_T&=\frac{2}{1+u_x^2}(u_x)_T\\
&=\frac{2}{1+u_x^2} (u_{xt}+u u_{xx})\\

&=\frac{2}{1+u_x^2}( -\frac{1}{2} u_x^2+u^{ 2}-P +kQ_x)\\

&=- \sin^2\frac v2+2 \cos^2\frac v2  (u^2-P +kQ_x ).
\end{array}\]
For $\xi_T$, from $Y_t+u Y_x=0$, we have $Y_{tx}+uY_{xx}=-u_xY_x$. Then,  
  (\ref{Balance1})  implies 
\[\begin{array}{lll}
\xi_T&=\frac{2u_x}{Y_x}  (u_{tx}+uu_{xx}) +\frac{-(1+u_x^2) }{Y_x^2}(Y_{tx}+uY_{xx})\\

&=\frac{1}{Y_x} (2u_xu_{tx}+2uu_xu_{xx}+u_x^3+u_x)\\

&=2\frac{1+u_x^2}{Y_x} \frac{u_x}{1+u_x^2}[\frac{(u_x^2)_t+(uu_x^2)_x}{2u_x}+\frac12]\\

&=\xi \sin v (\frac 12 +u^2-P+kQ_x).
\end{array}\]
In conclusion, from the new transformation (\ref{T}),     (\ref{E}) becomes
\be\label{S-1}
\left\{\begin{array}{lll}
&u_T=-P_x+kQ,\\

& v_T=- \sin^2\frac v2+2 \cos^2\frac v2  (u^2-P +kQ_x ),\\
 
& \xi_T=\xi \sin v (\frac 12 +u^2-P+kQ_x).
\end{array}\right.
\ee
We supplement (\ref{S-1}) with the initial data under the new coordinate $(T,Y)$
\be\label{S-2}
\left\{\begin{array}{lll}
&u(0,Y)=u_0(x_0(Y))\\

& v(0,Y)=2 \arctan (u_0)_x(x_0(Y)),\\
 
& \xi(0,Y)=1.
\end{array}\right.
\ee
Next, we show  the expression of $P$, $P_x$, $Q$ and $Q_x$ under the new coordinate $(T,Y)$. It follows from the last formula in (\ref{T2}) that 
\[x(T,Y)-x(T,Y^{'})=\int_{Y^{'}}^Y \cos^{2}\frac{v(T,s)}{2} \cdot \xi(t,s)ds.\]
We take $x=x(T,Y^{'})$, then $dx=\frac{\xi}{1+u_x^2}d Y^{'}$, the validity of which will be checked below. Thus,    $P$, $P_x$, $Q$ and $Q_x$ has the following form under  $(T,Y)$.
\be\label{P1}\begin{array}{lll}P(T,Y)&=P(T,x(T,Y))=p*[ u^{2} + \frac12 u_x^{2}]=\frac 12\int_{-\infty}^{+\infty}e^{-|x(T,Y)-x|}[u^{2} + \frac12 u_x^{2}]dx\\

&=\frac 12\int_{-\infty}^{+\infty}e^{-|\int_{Y}^{Y^{'}}\cos^2\frac{v(s)}{2} \cdot\xi(s)ds|} [ u^{2}\cos^2\frac{v(Y^{'})}{2} +\frac{1}{2}\sin^{2}\frac{v(Y^{'})}{2} ] \xi(Y^{'}) dY^{'},\\
\end{array}\ee
\be\label{P2}\begin{array}{lll}P_x(T,Y)&=P_x(T,x(T,Y))=p_x*[\frac{2\lambda+1 }{2}u^{\lambda}u_x^2+ u^{\lambda+2}]\\

&=\frac 12(\int_{Y}^{+\infty}-\int_{-\infty}^{Y})e^{-|\int_{Y}^{Y^{'}}\cos^2\frac{v(s)}{2} \cdot\xi(s)ds|}\\

&\qquad \qquad \qquad   \cdot [ u^{2}\cos^2\frac{v(Y^{'})}{2} +\frac{1}{2}\sin^{2}\frac{v(Y^{'})}{2} ] \xi(Y^{'}) dY^{'},
\end{array}\ee
\be\label{Q1} Q(T,Y)=\frac 12\int_{-\infty}^{+\infty}e^{-|\int_{Y}^{Y^{'}}\cos^2\frac{v(s)}{2} \cdot\xi(s)ds|} u\cos^2 \frac  {v(Y^{'})}{2}\xi(Y^{'}) dY^{'},\ee
\be\label{Q2} Q_x(T,Y)= \frac{1 }{2} (\int_{Y}^{+\infty}-\int_{-\infty}^{Y})e^{-|\int_{Y}^{Y^{'}}\cos^2\frac{v(s)}{2} \cdot\xi(s)ds|}   u\cos^2 \frac  {v(Y^{'})}{2}\xi(Y^{'}) dY^{'}.\ee

We remark that the equivalent semi-linear system (\ref{S-1}) is invariant under translation by $2\pi $ in $v$. It would be more precise to use $e^{iv}$ as variable. For simplicity, we use $v\in[-\pi,\pi]$ with endpoints identified.

\renewcommand{\theequation}{\thesection.\arabic{equation}}
\setcounter{equation}{0}
\section{Local Existence of the Transferred System}

In this section, we will prove the existence of weak solutions of the transferred system (\ref{S-1})-(\ref{S-2}) by the contracting mapping theory. Here, the work space $X$ is defined by 
$X:= H^1(\mathbb{R})  \times[L^2(\mathbb{R})\cap L^{\infty}(\mathbb{R})]\times L^{\infty}(\mathbb{R})$
with its norm $\|(u,v,\xi)\|_{X}=\|u\|_{H^1 }+\|v\|_{L^2}+\|v\|_{L^{\infty}}+\|\xi\|_{L^{\infty}}$.  As usual, we shall prove the existence of a fixed point of the integral transformation: $\Phi(u,v,\xi)=(\widetilde{u},\widetilde{v},\widetilde{\xi})$, where $(\widetilde{u},\widetilde{v},\widetilde{\xi})$ is defined by 
\[
\left\{\begin{array}{lll}
& \widetilde{u}(T,Y)=u_0(x_0(Y)+\int_0^{T}(-P_x(\tau,Y)+kQ(\tau,Y))d\tau,\\

& \widetilde{v}(T,Y)=2 \arctan (u_0)_x(x_0(Y))+\int_0^{T}[- \sin^2\frac v2+2 \cos^2\frac v2  (u^2-P +kQ_x )]d\tau,\\
  
&  \widetilde{\xi}(T,Y)=1+\int_0^{T}[\xi \sin v (\frac 12 +u^2-P+kQ_x)]d\tau,
\end{array}\right.
\]
where $P$, $P_x$, $Q$ and $Q_x$ are defined by (\ref{P1})-(\ref{Q2}).

As the standard ODE's theory in the Banach space, we shall prove that all functions on the right hand side of (\ref{S-1}) are locally Lipschitz continuous with respect to $(u,v,\xi)$ in $X$. Then, we can obtain the local existence of  weak solutions of (\ref{S-1})-(\ref{S-2}), as follows.
\begin{proposition}  Let $k\in \mathbb{R}$  be a constant.   If  
$u_0\in H^1(\mathbb{R})  $, then the Cauchy problem (\ref{S-1})-(\ref{S-2}) has a unique solution on the interval $[0, T]$.

\end{proposition}
\begin{proof}
Let $K\subset X$ be a bounded domain and defined by 
\[K =\{(u,v,\xi)| \ \|u\|_{H^1}\leq A,\ \|v\|_{L^2}\leq B,  \|v\|_{L^{\infty}}\leq \frac{3\pi}{2},  
 \xi(Y)\in[C^-, C^+] \}\]
  for a.e.  $Y\in \mathbb{R}$  and constants $A$, $B$, $C^-$, $C^+$ $>0$. Then, if  the mapping $\Phi(u,v,\xi)$ is Lipschitz continuous on $K$, then
$\Phi(u,v,\xi)$ has a fixed point by contraction argument and   the existence of local solutions will be followed. 

In order to do so, first, it follows from the Sobolev embedding $H^{1 }(\mathbb{R})\hookrightarrow L^{\infty}(\mathbb{R})$ with any $q>1$ that 
$\|u\|_{L^{\infty}}\leq C\|u\|_{H^1}$. 
Due to the uniform boundness of $v$ and $\xi$, the following maps:
\[    \sin^2\frac v2, \ u^{ 2}\cos^2\frac v2, \     \xi  \sin v,\  u^2 \xi \sin v\]
are all Lipschitz continuous as maps from $K$ into $L^2(\mathbb{R})$, as well as  from $K$ into $L^{\infty}(\mathbb{R})$. 
In order to estimate the singular integrals in $P$, $P_x$, $Q$ and $Q_x$, we observe that $e^{-|\int_{Y}^{Y^{'}}\cos^2\frac{v(s)}{2} \cdot\xi(s)ds|}\leq 1$. But this is not enough to control the above singular integrals. Here, we claim the following estimate:    If $\|v\|_{L^2}\leq B$, then for any $(u,v,\xi)\in K$, $f\in L^q(\mathbb{R})$ with $q\geq 1$
\be\label{Singular} \|\int_{-\infty}^{+\infty}  e^{-|\int_{Y}^{Y^{'}}\cos^2\frac{v(s)}{2} \cdot\xi(s)ds|} f(Y^{'})d Y^{'}\|_{L^q}\leq  \|g\|_{L^1} \ \|f\|_{L^q},\ee
where $g(z)=\min\{1,e^{ \frac{C^-}{2 }(\frac{B^2 }{2 }-|z|)}\}$ and $\|g\|_{L^1}= B^2+\frac{4}{ C^-}$. 
 Indeed, when $ \frac{\pi}{4}\leq |\frac{v(Y)}{2}|\leq  \frac{3\pi}{4}$, we get $\sin^2\frac{v(Y)}{2}\geq \frac12 $. Then, 
for any $(u,v,\xi)\in K$,  
\[{\rm meas }\{Y\in \mathbb{R}|\ |\frac{v(Y)}{2}|\geq \frac{\pi}{4}  \} \leq  2  \int_{\{Y\in \mathbb{R}| \sin^2\frac{v(Y)}{2}\geq \frac{1}{2}\}} \sin^2\frac{v(Y)}{2}dY \leq    \frac{1  }{2}\|v\|_{L^2}^2.\]
For any $z_1<z_2$, we deduce that     
\[\begin{array}{lll}\int_{z_1}^{z_2}\cos^{2}\frac{v(z)}{2} \cdot\xi(z)dz

&\geq \frac{ C^-}{2 }(|z_2-z_1|- \int_{\{ z\in[z_1,z_2]|  |\frac{v(Y)}{2}|\geq \frac{\pi}{4} \}}dz)\\

&\geq  \frac{ C^-}{2 }(|z_2-z_1|-  \frac{1}{2}B^2 ).\end{array} \]
The above estimate guarantees proper control on the singular integrals in $P$, $P_x$, $Q$ and $Q_x$, which decreases quickly as $|Y-Y^{'}|\rightarrow +\infty$. Therefore, taking $g(z)=\min\{1,e^{ \frac{C^-}{2 }(\frac{B^2 }{2 }-|z|)}\}$, for every $q\geq 1$, we see that 
\[ |\int_{-\infty}^{+\infty}  e^{-|\int_{Y}^{Y^{'}}\cos^2\frac{v(s)}{2} \cdot\xi(s)ds|} f(Y^{'})d Y^{'}|\leq |g*f(Y)|.\] 
(\ref{Singular}) follows from the Young inequality. In addition, 
\[ \|g\|_{L^1} =\int_{-  \frac{B^2}{2 } }^{ \frac{B^2}{2 } } 1 dz+\int_{ \frac{B^2}{2} }^{+\infty} e^{  \frac{C^-}{2}(\frac{B^2}{2 } - z) }dz  +\int^{- \frac{B^2}{2 } }_{-\infty} e^{  \frac{C^-}{2 }( \frac{B^2}{2 } + z) }dz=  B^2+\frac{4}{ C^-}. \]
This completes the proof of Claim (\ref{Singular}).

Thirdly, by (\ref{Singular}), we give a priori bounds on $P$, $P_x$, $Q$ and $Q_x$, which implies $P$, $P_x$, $Q$ and $Q_x$ $\in H^{1}(\mathbb{R})$. More precisely, we deduce that 
 $P$, $\partial_YP$, $P_x$, $\partial_YP_x$, $Q$, $\partial_YQ$, $Q_x$, $\partial_YQ_x$ $\in L^2(\mathbb{R})$.
Indeed, we give estimates for $Q$, and the estimates for $P$ can be obtain by similar argument. By applying (\ref{Singular}), we deduce that 
{\small
\be\label{Q-1} \partial_YQ  =\frac 12 (\int_{Y}^{+\infty}-\int^{Y}_{-\infty}) 
e^{-|\int_{Y}^{Y^{'}}\cos^2\frac{v(s)}{2} \cdot\xi(s)ds|} {\rm sign} (Y^{'}-Y)   \cdot[ u  \cos^{4}\frac{v(Y^{'})}{2}] \cdot \xi^2(Y^{'}) dY^{'},\ee
\be\label{Q-2}\begin{array}{lll}\partial_YQ_x &=-u\xi\cos^2 \frac{v(Y^{'})}{2}+
\frac{1 }{2} \int_{-\infty}^{+\infty}e^{-|\int_{Y}^{Y^{'}}\cos^2\frac{v(s)}{2} \cdot\xi(s)ds|}  \cdot[ u  \cos^{4}\frac{v(Y^{'})}{2}] \cdot \xi^2(Y^{'}) dY^{'}.\end{array}\ee
}
Then, we have 
\be\label{QL2-1}\|Q\|_{L^2}\leq \frac {C^+}{2} \|g\|_{L^1} \|u\|_{L^2},\ \ \|\partial_YQ\|_{L^2}\leq (C^+)^2 \|g\|_{L^1} \|u\|_{L^2},\ee
 \be\label{QL2-2}\|Q_x\|_{L^2}\leq C^+  \|g\|_{L^1} \|u\|_{L^2},\ \ \|\partial_YQ_x\|_{L^2}\leq(C^++ \frac{(C^+)^2}{2} \|g\|_{L^1}) \|u\|_{L^2}. \ee
By the same arguments, we can obtain the similar estimates for $P$ in the following. 
 \be\label{PL2-1}\begin{array}{lll}
 \vspace{0.1cm}
 &\|P\|_{L^2}\leq \frac {C^+}{2} \|g\|_{L^1} (\|u\|_{L^\infty}\|u\|_{L^2}+\frac 18\|v\|_{L^2}),\\
  \vspace{0.1cm}
 &\|\partial_YQ\|_{L^2}\leq (C^+)^2   \|g\|_{L^1} (\|u\|_{L^\infty}\|u\|_{L^2}+\frac 18\|v\|_{L^2}),\\
  \vspace{0.1cm}
 &\|P_x\|_{L^2}\leq C^+     \|g\|_{L^1} (\|u\|_{L^\infty}\|u\|_{L^2}+\frac 18\|v\|_{L^2}),\\
 
&\|\partial_YP_x\|_{L^2}\leq(C^++ (C^+)^2  \|g\|_{L^1}) (\|u\|_{L^\infty}\|u\|_{L^2}+\frac 18\|v\|_{L^2}). \end{array}\ee
Using the Sobolev inequality, we have $\|u\|_{L^\infty}\leq C\|u\|_{H^1}$ and this implies $P$, $P_x$, $Q$ and $Q_x$ $\in H^{1}(\mathbb{R})$. Then, according to  the above discussion, we show that the mapping $\Phi (u,v,\xi)$ is from $K$ to $K$.

Fourthly,   we shall prove the mapping $\Phi (u,v,\xi)$ is Lipschitz continuity with respect to $(u,v,\xi)$. More precisely, we need to prove that the following partial derivatives
$\frac{\partial P}{\partial u}, \ \frac{\partial P}{\partial v}, \ \frac{\partial P}{\partial \xi}, \ \frac{\partial P_x}{\partial u},\ \frac{\partial P_x}{\partial v},\ \frac{\partial P_x}{\partial \xi}$, 
$\frac{\partial Q}{\partial u}, \ \frac{\partial Q}{\partial v}, \ \frac{\partial Q}{\partial \xi}, \ \frac{\partial Q_x}{\partial u},\ \frac{\partial Q_x}{\partial v},\ \frac{\partial Q_x}{\partial \xi}$
are uniformly bounded as $(u,v,\xi)$ range inside the domain $K$. We observe these derivatives are bounded linear operators from the appropriate spaces into $H^{1 }(\mathbb{R})$. For sake of illustration, we shall give the detail estimates for $Q$. And the estimates for $P$ will be similar obtained.
 In other words, we shall prove $ \frac{\partial Q}{\partial u}$ and $\frac{\partial Q_x}{\partial u}$ are linear operators from $H^1(\mathbb{R})$ into $H^1(\mathbb{R})$;  $\frac{\partial Q}{\partial v}$ and $\frac{\partial Q_x}{\partial v}$ are linear operators from $L^2(\mathbb{R})\bigcap L^{\infty}(\mathbb{R})$ into $L^2(\mathbb{R})$; $\frac{\partial Q}{\partial \xi}$ and $\frac{\partial Q_x}{\partial \xi}$ are linear operators from $  L^{\infty}(\mathbb{R})$ into $L^2(\mathbb{R})$. 
 
 Here, we shall take $ \frac{\partial Q}{\partial u}$ and $ \frac{\partial (\partial_YQ)}{\partial u}$ as examples to illustrate the main idea. Indeed,  $\forall\ \phi \in H^1(\mathbb{R})$,
  \[\begin{array}{lll}\vspace{0.1cm}
  & \|\frac{\partial Q}{\partial u}\cdot \phi\|_{L^2}\leq \frac 12 \|\xi\|_{L^\infty}\|g*\phi\|_{L^2}\leq \frac {C^+}{2} \|g\|_{L^1}  \|\phi\|_{H^1},\\
  \vspace{0.1cm}
  & \|\frac{\partial (\partial_YQ)}{\partial u}\cdot \phi\|_{L^2}\leq  \|\xi\|_{L^\infty}^2\|g*\phi\|_{L^2}\leq  (C^+)^2  \|g\|_{L^1}  \|\phi\|_{H^1},\\
\vspace{0.1cm}
  
 &  \|\frac{\partial Q_x}{\partial u}\cdot \phi\|_{L^2}\leq   \|\xi\|_{L^\infty}\|g*\phi\|_{L^2}\leq  C^+ \|g\|_{L^1}  \|\phi\|_{H^1},\\
  
  & \|\frac{\partial (\partial_YQ_x)}{\partial u}\cdot \phi\|_{L^2}\leq \|\xi\|_{L^\infty}\| \phi\|_{L^2}+\frac12 \|\xi\|_{L^\infty}^2\|g*\phi\|_{L^2}\leq (C^++\frac{(C^+)^2}{2}  \|g\|_{L^1} ) \|\phi\|_{H^1}.  \end{array}\]
 Thus, the linear operator $\frac{\partial Q}{\partial u}$, $\frac{\partial (\partial_YQ)}{\partial u}$, $\frac{\partial Q_x}{\partial u}$ and $\frac{\partial (\partial_YQ_x)}{\partial u}$ are bounded from $H^1(\mathbb{R})$ into $L^2(\mathbb{R})$.

Finally, we prove the mapping $\Phi (u,v,\xi)$ is uniformly Lipschitz continuity on a neighborhood $K$ of the initial data in the space $X$. Then, we can apply the standard theory of ODE's local existence in Banach spaces, there exist a solutions to the Cauchy problem (\ref{S-1})-(\ref{S-2}) on small time interval $[-T,T]$.  This completes the proof of Proposition 3.1.
\end{proof}

\renewcommand{\theequation}{\thesection.\arabic{equation}}
\setcounter{equation}{0}
\section{Global Existence of the Transferred System}

In this section, we shall prove that the local solution  for the Cauchy problem (\ref{S-1})-(\ref{S-2}) can be extended  to the global one. To ensure this, we need to show that there exists a   bound $C(E_0)>0$ such that 
\be\label{B}\|u(T)\|_{H^{1}}+\|v(T)\|_{L^2}+\|v(T)\|_{L^{\infty}}+\|\xi(T)\|_{L^{\infty}}+\|\frac{1}{\xi(T)}\|_{L^{\infty}}\leq C(E_0),\ee
for all $T\in \mathbb{R}$. Then, we obtain the following proposition. 
\begin{proposition} \label{Global}  Let $k\in \mathbb{R}$  be a constant.  If the initial data 
$u_0\in H^{1}(\mathbb{R})$, then the Cauchy problem (\ref{S-1})-(\ref{S-2}) has a unique solution, defined for all time  $T\in \mathbb{R}$.
\end{proposition}
\begin{proof} In Proposition 3.1, we have proved the existence of the local solution  for the Cauchy problem (\ref{S-1})-(\ref{S-2}).  
  First of all, when $t=0$, from the transformation (\ref{T})-(\ref{T2}), we see that  $Y_x(0)= 1+(u_0)_x^2$, then 
\[u_Y(0,Y)=\frac{(u_0)_x}{ 1+(u_0)_x^2 }=\frac12  \sin v(0,Y)  \ \ {\rm and }\ \ \xi(0,Y)=1.\] We claim that  for all time  $T$,
\be\label{uY}u_Y=\frac12 \xi \sin v=\xi \sin \frac v2 \cos \frac v2  .\ee
Indeed, from the first identity in (\ref{S-1}),   we have 
\[\begin{array}{lll}
\vspace{0.1cm}
(u_Y)_T=(u_T)_Y&=-(P_x)_Y+kQ_Y=\frac{-P_{xx}+kQ_x}{Y_x}\\

&=\xi [ u^{2}\cos^2\frac{v}{2}+\frac12 \sin^2\frac v2 +\cos^2\frac v2 (-P+kQ_x)].\end{array}\]
From the last two identity in (\ref{S-1}), we see that for all  time   $T$
\[\begin{array}{lll}
\vspace{0.1cm}

 (\frac 12 \xi \sin v )_T&= \frac 12 \sin v\  \xi_T+\frac 12 \xi \cos v\ v_T\\
\vspace{0.1cm}

&= 2u^2\xi \sin^2 \frac v2 \cos^2 \frac v2+2\xi \sin^2 \frac v2\cos^2 \frac v2 (-P+kQ_x)\\
\vspace{0.1cm}
&\qquad +\frac 12 \xi \sin^2 \frac v2+2u^2\xi\cos^4 \frac v2+2\xi \cos^4\frac v2 (-P+kQ_x)\\
\vspace{0.1cm}
&\qquad -\xi u^2\cos^2 \frac v2-\xi \cos^2 \frac v2 (-P+kQ_x)\\

&=(u_Y)_T.
\end{array}\]
Then, it follows from    $\xi(0,Y)=1$  that (\ref{uY}) is true for all time$T$.  Next, we  estimate the functional  $E(T):=\int_{\mathbb{R}}(u^2\xi \cos^{2} \frac {v}{2}+\xi\sin^2\frac {v}{2})\ dY$ under the new coordinate. From (\ref{S-1}), we have
\be\label{new-energy}\begin{array}{lll}\vspace{0.1cm}
&\frac{d}{dT}\int_{\mathbb{R}}(u^2\xi \cos^{2} \frac {v}{2}+\xi\sin^2\frac {v}{2})\ dY\\
\vspace{0.1cm}
=&\int_{\mathbb{R}}\{(u^2\cos^{2} \frac {v}{2}+\sin^2\frac {v}{2})\xi_T+2u\xi\cos^2 \frac v2 \ u_T +\xi \sin\frac v2\cos\frac v2(1-u^2)v_T \}dY\\
\vspace{0.1cm}

=& \int_{\mathbb{R}}\{u^2\xi\sin \frac v2\cos^3 \frac v2+2u^4\xi \sin \frac v2 \cos^3\frac v2+2u^2\xi \sin\frac v2 \cos^3\frac v2 (-P+kQ_x)\\
\vspace{0.1cm}

&+\xi\sin^3\frac v2 \cos \frac v2+2u^2\xi \sin^3\frac v2\cos \frac v2+2\xi\sin^3\frac v2\cos\frac v2 (-P+kQ_x)\\
\vspace{0.1cm}

&+2u\xi \cos^2 \frac v2 (-P_x+kQ)-\xi\sin^3\frac v2\cos\frac v2+2u^2\xi \sin\frac v2\cos^3\frac v2\\
\vspace{0.1cm}
&+2\xi\sin\frac v2\cos^3\frac v2 (-P+kQ_x)+u^2\xi\sin^3\frac v2\cos\frac v2-2u^4\xi\sin\frac v2 \cos^3\frac v2\\
\vspace{0.1cm}

 &-2u^2\xi \sin\frac v2\cos^3\frac v2 (-P+kQ_x)\}dY\\

\vspace{0.1cm}

=&\int_{\mathbb{R}}\{3u^2\xi \sin\frac v2 \cos \frac v2+2u\xi \cos^2\frac v2(-P_x+kQ)+2\xi\sin\frac v2 \cos\frac v2(-P+kQ_x)\}dY\\

\vspace{0.1cm}

=&\int_{\mathbb{R}} 3u^2u_YdY+2 \int_{\mathbb{R}}[u(-P+kQ_x)]_YdY+2k\int_{\mathbb{R}}u^2\xi \cos^2 \frac v2 dY\\

=&2k\int_{\mathbb{R}}u^2\xi \cos^2 \frac v2 dY,
\end{array}\ee
where in the last two estimates, we use the facts: for any $u\in H^{1}(\mathbb{R})$, $\lim\limits_{|Y|\rightarrow +\infty} u(Y)=0$,
$(u^3)_Y=3u^2 u_Y=3u^2\xi \sin \frac v2 \cos \frac v2$ and 
\[2[u(-P+kQ_x)]_Y=2\xi \sin \frac v2\cos \frac v2(-P+kQ_x)+2u\xi\cos^2 \frac v2(-P_x+kQ)+2ku^2\xi\cos^2\frac v2.\]
Then, from (\ref{new-energy}), we can deduce that there exists a constant  $D_0:=D_0(E_0,T)>0$ such that the new energy $E(T)$ has a priori 
 on bounded intervals of
time by the Gronwall inequality. That is,  for any $T$ on bounded intervals,
\be\label{energy}E(T)\leq D_0^2.\ee

Secondly, we estimate $\|u(T)\|_{L^{\infty}}$, $\|\xi(T)\|_{L^{\infty}}$, $\|\frac{1}{\xi(T)}\|_{L^{\infty}}$ and  $\|v(T)\|_{L^{\infty}}$ with respect to $T$. For all the solutions of the Cauchy problem (\ref{S-1})-(\ref{S-2}), it follows from (\ref{uY}) and (\ref{energy}) that 
\[\begin{array}{lll}
\vspace{0.1cm}
\sup\limits_{Y\in \mathbb{R}} |u^2(T,Y)|&\leq \int_{\mathbb{R}}|(u^2)_Y|dY=\int_{\mathbb{R}} |u \sin \frac v2 \cos  \frac v2|\ \xi dY\\

&\leq \int_{\mathbb{R}}(u^2\xi\cos^{2 } \frac {v}{2}+\xi\sin^2\frac {v}{2}) dY,

\end{array}\]
which implies that $\|u(T)\|_{L^{\infty}}$ has a   priori bound $D_0$. More precisely, we have the following estimate:
\be\label{energy2}\|u(T)\|_{L^{\infty}}\leq D_0.\ee
By the definition of $P$,  $P_x$, $Q$ and $Q_x$ in (\ref{P1})-(\ref{Q2}), using (\ref{energy}) and (\ref{energy2}), we deduce that 
\be\label{energy3}  \|P(T)\|_{L^{\infty}}, \|P_x(T)\|_{L^{\infty}} \leq  \|u^2\xi \cos^2 \frac v2+\frac 12 \xi\sin^2\frac v2\|_{L^1} \leq D_0^2,\ee
\be\label{energy4} \|Q(T)\|_{L^{\infty}}, \|Q_x(T)\|_{L^{\infty}}\leq\|e^{-|\int_{Y}^{Y^{'}} \cos^{2}\frac v2 \cdot \xi ds|} \xi(T)\|_{L^{1}}\|u\|_{L^\infty}\leq  C D_0.\ee
Inject (\ref{energy3}) and (\ref{energy4}) into the last identity in (\ref{S-1}), we see that 
\[|\xi_T|\leq  \|\frac12+u^2-P+kQ_x\|_{L^\infty}\xi\leq (\frac 12 +2D_0^2+kCD_0)\xi.\]
And from the initial condition $\xi(0,Y)=1$, we get
\be\label{Bound1} \frac{1}{D_1}\leq \xi (T)\leq D_1,\ee  where $D_1=e^{(\frac 12 +2D_0^2+kCD_0)|T|}$. From the second  identity in (\ref{S-1}), 
we have
\[\frac{d}{dT}\|v(T)\|_{L^{\infty}}\leq 2(\|u(T)\|_{L^{\infty}}^{2} + \|P \|_{L^{\infty}}+ k\|Q_x \|_{L^{\infty}})\leq 2D_0( kC+2D_0).\]
Denote $D_2=2D_0(kC+2D_0)$. Then,  we have
\be\label{Bound2}  \|v(T)\|_{L^{\infty}} \leq e^{D_2 T}.\ee

Thirdy, we estimate $\|u(T)\|_{H^{1}}$ with respect to $T$. From $u_T=-P_x+kQ$,  
\be\label{1} \frac {d}{dT}\| u(T)\|_{L^{2}}^{2 } =2 \int_{\mathbb{R}} u  u_TdY\leq  2 \|u(T)\|_{L^{\infty}} (\|P_x\|_{L^1}+ k\| Q\|_{L^1}),\ee
\be\label{2} \frac {d}{dT}\| (u(T))_Y\|_{L^{2}}^{2}=2 \int_{\mathbb{R}} u_Y  (u_T)_YdY \leq 2 \|u_Y\|_{L^{\infty}} (\|(P_x)_Y\|_{L^1}+ k\|Q_Y\|_{L^1}).\ee
In order to obtain the boundness of  $\|u(T)\|_{H^{1}}$, we need the following lemma to estimate $\|P_x\|_{L^1}$,  $\| Q\|_{L^1}$, $\|(P_x)_Y\|_{L^1}$ and $\|Q_Y\|_{L^1}$.
\begin{lemma} \label{Lemma} Let $f\in L^q$ with $q\geq 1$. If $\|\xi \sin^2\frac {v(T)}{2} \|_{L^1}\leq D_0^2$, then 
\be\label{Singular2} \|\int_{-\infty}^{+\infty}  e^{-|\int_{Y}^{Y^{'}}\cos^2\frac{v(s)}{2} \cdot\xi(s)ds|} f(Y^{'})d Y^{'}\|_{L^q}\leq  \|h*f\|_{L^q}\leq\|h\|_{L^1} \ \|f\|_{L^q},\ee where $h(z)=\min\{1, e^{\frac{1}{2D_1}(  2D_1D_0^2-|z|)}\}$ and $\|h\|_{L^1}=4D_1D_0^2+4D_1$.
\end{lemma}
\begin{proof} We remark that $\|\xi \sin^2\frac {v(T)}{2}  \|_{L^1}\leq D_0^2$ does not imply that $\|v(T)\|_{L^2}$ is bounded, which is the main difficulty. As proved in above, we see that $\|v(T)\|_{L^\infty}$ is bounded, and we may assume $|v|\leq \frac {3\pi}{2}$ because $v(T)$ is invariant under multiplier of $2\pi$. Then,   when $\frac{\pi}{4}\leq |\frac{v(Y)}{2}|\leq  \frac{3\pi}{4}  $, we have $\sin^2   \frac{v(Y)}{2}\geq \frac12$.  And,  
we deduce that  
\be\label{Measure}\begin{array}{lll}& {\rm meas}\{Y\in \mathbb{R}|\ |\frac{v(Y)}{2}|\geq \frac{ \pi}{4}\}\\
\leq & {\rm meas} \{Y\in \mathbb{R}| \ \sin^2\frac{v(Y)}{2}\geq \frac 12\}\\

\leq& 2 \int_{\{Y\in \mathbb{R}| \sin^2\frac{v(Y)}{2}\geq \frac 12 \}}D_1\ \xi   \sin^2\frac{v(Y)}{2}dY\\

\leq&  2 D_1D_0^2.
\end{array}\ee
 For any $z_1<z_2$,  we have
\[\begin{array}{lll}\int_{z_1}^{z_2}\cos^{2}\frac{v(z)}{2} \cdot\xi(z)dz&\geq \frac{ 1}{2D_1} ( 
\int_{z_1}^{z_2}1 dz-\int_{\{ z\in [z_1,z_2]\ |\   |\frac{v(Y)}{2}|\geq \frac {\pi}{4} \}} 1 dz)\\

&\geq \frac{1}{2D_1}(|z_2-z_1|-  2 D_1D_0^2 ). \end{array}\]
 Let $h(z)=\min\{1, e^{\frac{1}{2D_1}(  2D_1D_0^2-|z|)}\}$. Then,  for every $q\geq 1$, (\ref{Singular2}) follows from the  Young inequality. In addition, by some computation, we have 
 \[\begin{array}{lll}
 \vspace{0.1cm}
 \|h\|_{L^1}&=\int_{- 2D_1D_0^2} ^{  2D_1D_0^2} 1 dz+\int_{ 2D_1D_0^2}^{+\infty} e^{\frac{2D_1D_0^2-z}{2D_1}}dz +\int^{- 2D_1D_0^2}_{-\infty} e^{\frac{ 2D_1D_0^2+ z }{2D_1}}dz\\

&= 4D_1D_0^2+4D_1.\end{array}\]
  
\end{proof}
Now, we return to the proof of Proposition  \ref{Global}. By using Lemma \ref{Lemma}, we have the following estimates:
\be\label{PQL1}\begin{array}{lll}\vspace{0.1cm}

& \|P_x\|_{L^1}\leq D_0^2\|h\|_{L^1},\ \ \ \ \ &\|(P_x)_Y\|_{L^1}\leq (1+D_1)D_0^2 \|h\|_{L^1},\\

& \| Q\|_{L^1}\leq D_1D_0\|h\|_{L^1}, \ \ \ \ & \|Q_Y\|_{L^1}\leq D_1^2D_0\|h\|_{L^1}.\end{array}\ee
Injecting (\ref{PQL1}) into (\ref{1}) and (\ref{2}),   we can deduce that $\| u(T)\|_{L^{2 }}^{2 }$ and $\| (u(T))_Y\|_{L^{2 }}^{2 }$ are bounded on any bounded interval of time $T$.    That is, there exists a constant $D_3:=D_3(E_0,T)>0$  such that 
\be\label{Bound3}  \|u(T)\|_{H^{1}}  \leq D_3.\ee

Finally, we estimate $\|v(T)\|_{L^{2}}$ with respect to $T$. Multiplying $v(T)$ to the second identity in (\ref{S-1}) and integrating, we deduce that 
\be\label{V1}\begin{array}{lll}
\vspace{0.1cm}
\frac{d}{dT}\|v(T)\|_{L^{2}}^2&=-2\int_{\mathbb{R}} \sin^2\frac v2 vdY+4\int_{\mathbb{R}}\cos^2 \frac v2(u^2-P+kQ_x)vdY\\
\vspace{0.1cm}
&\leq 2D_1\|v(T)\|_{L^\infty}\| \xi \sin^2 \frac{v(T)}{2}\|_{L^1}\\

&\qquad +4(\|u(T)\|_{L^\infty}\|u(T)\|_{L^2}+\|P\|_{L^2}+k\|Q_x\|_{L^2})\|v(T)\|_{L^2}.\end{array}\ee
Inject (\ref{energy}), (\ref{energy2}), (\ref{Bound1}), (\ref{Bound2}) and (\ref{Bound3}) into the estimates in $\|P\|_{L^2}$ and $\|Q_x\|_{L^2}$.  We have
\be\label{V2} \|P\|_{L^2}^2\leq D_1(\frac 12 +D_0^2)D_0^2\|h\|_{L^1}^2, \  \ \  \|Q_x\|_{L^2}^2\leq D_1^2 D_0^2 \|h\|_{L^1}^2.\ee
Then, injecting (\ref{V2}) into (\ref{V1}) and using the Gronwall inequality,  
we can deduce that $\| v(T)\|_{L^{2 }}^{2 }$  is bounded on any bounded interval of time $T$.    That is, there exists a constant $D_4:=D_4(E_0,T)>0$  such that 
\be\label{Bound4}  \|v(T)\|_{L^{2}}  \leq D_4.\ee
This completes the proof of Proposition \ref{Global}.
\end{proof}
 \begin{remark}
We defined the set of times
\[N:= \{ T\geq 0\ | \ {\rm meas} \{ Y\in \mathbb{R}; \ v(T,Y)=-\pi\}>0\}.\] Then, we claim that the measure of $N$ must be $0$. That is $ {\rm meas}\ (N)=0$. 
Indeed, from the second identity in the  semi-linear system (\ref{S-1}), we see that 
$v_T=-1$ provided 
   $\cos v=-1$,   Then, it follows from the   the absolute continuity of $v$ that  we can find $\delta>0$ such that $v_T\leq M<0$ wherever $1+\cos v<\delta$. On the other hand, we remark the fact that $\|v(T)\|_{L^2}$ remains bounded on bounded time intervals, and then we can obtain $ {\rm meas}\ (N)=0$. If not, then we have 
$\int\int_{\{ v(T,Y)=-\pi \}} v_TdYdT<0$. 
This is contradict to the fact $v_T=0$ a.e. on $\{ v(T,Y)=-\pi \}$, in terms of  the absolute continuity of the map $T\rightarrow v(T,Y)$ at every fixed $Y\in \mathbb{R}$.
\end{remark}

 \renewcommand{\theequation}{\thesection.\arabic{equation}}
\setcounter{equation}{0}
\section{Existence of the Global  Weak Solution}

Now, we start with a global solution $(u,v,\xi)$ to (\ref{S-1}) obtained in Proposition 4.1. We define $x$ and $t$ as functions of $T$ and $Y$, where $t=T$ and 
\be\label{5.1}x(T,Y):=x_0(Y)+\int_0^Tu (\tau,Y)d\tau.\ee 
For each fixed $Y$, the function $T\mapsto x(T,Y)$ thus provides a solution to the Cauchy problem 
\[\frac{d}{dT} x(T,Y)=u (T,x(T,Y)), \ \ \ x(0,Y)=x_0(Y).\]
Then,  by  taking \be\label{5.2}u(T,x):=u(T,Y)\ \ \ {\rm provided }\ \ \  x(T,Y)=x,\ee we can prove that $u(t,x)$ is a solution of Eq.(\ref{E}). Then, we can complete   the proof of Theorem \ref{main-theorem}.   
 \begin{proof} 
First, we  prove that the function $u=u(t,x)$ is well-defined. From (\ref{energy2}), we see that 
$|u(T,Y)| \leq D_0 $. By (\ref{5.1}), we get 
\[x_0(Y)-D_0 T\leq x(T,Y)\leq x_0(Y)+D_0 T.\]
From  $Y=\int_0^{x_0(Y)} (1+(u_0)_x^2) dx$, we have $\lim\limits_{Y\rightarrow \pm\infty} x_0(Y) =\pm\infty$, this yields the image of the   map $(T,Y)\mapsto(T,x(T,Y))$
is the entire plane $\mathbb{R}^2$.  We claim  
\be\label{5.4} x_Y=\cos^{2} \frac{v}{2} \cdot \xi \ee for all $t$ and a.e. $Y\in \mathbb{R}$. 
Indeed, from   (\ref{S-1}), we deduce that 
\[\begin{array}{lll}&\frac{d}{dT} \cos^{2} \frac{v}{2} \cdot \xi\\

=& \cos^2 \frac v2\cdot \xi_T  -\xi\cos\frac v2 \sin\frac v2\cdot v_T\\

=& \xi \cos^3\frac v2\sin \frac v2+2u^2\xi \cos^3\frac v2 \sin \frac v2 +2\xi \cos^3\frac v2 \sin \frac v2 (-P+kQ_x)\\

&\quad  +\xi\sin^3\frac v2 \cos \frac v2-2u^2\xi \sin \frac v2 \cos^3\frac v2-2\xi\sin\frac v2\cos^3\frac v2 (-P+kQ_x)\\

=&  \xi \sin \frac v2 \cos  \frac v2\\

=&u_Y.
\end{array}\]
And by differentiating (\ref{5.1}) with $T$ and $Y$, we deduce that 
\[ \frac{d}{dT} x_Y=\frac{d}{dT}((x_0)_Y+\int_0^T   u_Yd\tau)=  u_Y=\frac{d}{dT} \cos^{2} \frac{v}{2} \cdot \xi.\]
Moreover, from the fact that $x\mapsto 2\arctan (u_0)_x(x)$ is measurable, we see that the claim (\ref{5.4})  is true for almost every $Y\in \mathbb{R}$ at $T=0$. Then, (\ref{5.4}) remains true for all times $T\in \mathbb{R}$ and a.e. $Y\in \mathbb{R}$.  Then, 
 for any $Y_1\neq Y_2$(without loss of generality, we assume $Y_1<Y_2$), if $x(t^*, Y_1)=x(t^*, Y_2)$, then, from the monotonicity of $x(t,Y)$ on $Y$,  
for every $Y\in [Y_1,Y_2]$, we have $x(t^*, Y )=x(t^*, Y_1)$. And,
from (\ref{5.4}) we get 
\[0=x(t^*, Y_1)-x(t^*, Y_2)=\int_{Y_1}^{Y_2} x_Y(t^*,Y)dY=\int_{Y_1}^{Y_2} \cos^{2} \frac{v(t^*,Y)}{2} \cdot \xi(t^*,Y)dY.\]
Then, $\cos \frac{v(t^*,Y)}{2} \equiv 0$ for every $Y\in [Y_1,Y_2]$.  Inject this into  (\ref{uY}).
\[\begin{array}{lll}u(t^*, Y_1)-u(t^*, Y_2)&=\int_{Y_1}^{Y_2} u_Y(t^*,Y)dY\\

&=\frac12 \int_{Y_1}^{Y_2}  \xi(t^*,Y) \sin \frac{v(t^*,Y)}{2}\cos  \frac {v(t^*,Y)}{2}dY=0.\end{array}\]
This proves that the map $(t,x)\mapsto u(t(T),x(T,Y))$ at (\ref{5.2}) is well defined for all $(t,x)\in \mathbb{R}^2$.

Secondly, we prove the regularity of $u(t,x)$ and energy equation. 
From   (\ref{energy}), we see that for $t$ in any bounded interval, 
\be\label{5.7}\begin{array}{lll}

& \int_{\mathbb{R}} (u^2(t,x)+u_x^2(t,x))dx\\

=&\int_{\{\mathbb{R}\cap \cos  v \neq -1\}}[u^2(T,Y)\cos^{2 } \frac {v(T,Y)}{2}+\sin^2\frac {v(T,Y)}{2} ]\ \xi (T,Y)dY\\

\leq & D_0^2.\end{array}\ee
Now, applying the Sobolev inequality: $\|u\|_{C^{0,\gamma}}\leq C \|u\|_{H^1} $,  where $\gamma=1-\frac12$.  From (\ref{5.7}), we get $\|u\|_{C^{0,\frac12}}\leq C \|u\|_{H^1}\leq CD_0 $, which implies $u$ of $x$ is    H\"{o}lder continuous  with index $\frac12$. On the other hand, it follows from the first identity in (\ref{S-1}), (\ref{energy3}) and (\ref{energy4})  that $\|u_t\|_{L^{\infty}}\leq C(\|P_x\|_{L^{\infty}}+\|Q\|_{L^{\infty}})\leq C(D_0)$. Then, the map $t\mapsto  u(t,x(t))$ is   Lipschitz continuous along every characteristic curve $t\mapsto x(t)$. Therefore, $u=u(t,x)$ is   H\"{o}lder continuous on the any bounded interval of times.

Thirdly, we prove that the $L^2(\mathbb{R})$-norm of $u(t)$  is Lipchitz continuous with respect to $t$ on any bounded interval. 
Denote $[\tau, \tau+h]$ to be any small interval and $\tau$ in any bounded interval.  For a given point $(\tau,\bar x)$, we choose the 
characteristic $t\mapsto x(t,Y)$ $:\{T\rightarrow x(T,Y)\}$ passes through the point $(\tau, \bar x)$, i.e. $x(\tau) =\bar x$. Since
the characteristic speed $u $ satisfies $\|u \|_{L^{\infty}}\leq C \|u \|_{H^1}\leq CD_0 $, we have the following estimate.
\[\begin{array}{lll}&|u(\tau+h,\bar x)-u(\tau,\bar x)|\\
\leq& |u(\tau+h,\bar x)-u(\tau+h,x(\tau+h,Y))|\\

&\qquad \qquad+|u(\tau+h,x(\tau+h,Y))-u(\tau,x(\tau, Y)|\\

\leq & \sup\limits_{|y-x|\leq CD_0 h} |u(\tau+h,y)-u(\tau+h,\bar x)| +\int_{\tau}^{\tau+h} |P_x(t,Y)|+k|Q(t,Y)|dt

\end{array}\]
Then,   we use   $(\int_a^b f(x)dx )^2 \leq [\sqrt{b-a} \|f\|_{L^2}]^2$ for $a<b$, 
and deduce that \be\label{L2}\begin{array}{lll}\vspace{0.1cm}
&\int_{\mathbb{R}}|u(\tau+h,\bar x)-u(\tau,\bar x)|^2dx\\
\vspace{0.1cm}
\leq & C   \int_{\mathbb{R}}(\int_{\bar x-CD_0 h}^{\bar x+CD_0 h} |u_x(\tau+h,y)|dy)^2dx\\
\vspace{0.1cm}
&\qquad \quad +C\int_{\mathbb{R}}(\int_{\tau}^{\tau+h} |P_x(t,Y)|+k|Q(t,Y)|dt)^2  \cdot \xi(\tau,Y)dY\\

 \vspace{0.1cm}

\leq &C   \int_{\mathbb{R}} 2D_0 h \int_{\bar x-CD_0 h}^{\bar x+CD_0 h} |u_x(\tau+h,y)|^2dy) dx\\
\vspace{0.1cm}
&\qquad \quad +C\int_{\mathbb{R}}(h \int_{\tau}^{\tau+h} |P_x(t,Y)|^2+k|Q(t,Y)|^2dt)   \|\xi(\tau)\|_{L^\infty}dY\\

\leq &4CD_0^2 h^2  \|u_x(\tau)\|_{L^2}^2 +Ch^2\|\xi(\tau)\|_{L^\infty} (\|P_x   \|_{L^2}^2+k\|Q \|_{L^2}^2)\\

\leq & C(D_0)h^2.
\end{array}\ee
Then,  
This implies that the map $t\mapsto u(t)$ is Lipchitz continuous, in terms of the $x$-variable in $L^2(\mathbb{R})$.

Fourthly, we prove  that the function $u$ provides a weak solution of (\ref{E}). Denote
$\Gamma:=\{(t,x)| t\in \mathbb{R},\ x\in \mathbb{R}\}$. 
 For any test function $\phi(t,x)\in C_c^1(\Gamma)$, the first equation of (\ref{S-1}) has the following weak form. 
\[\begin{array}{lll}\vspace{0.1cm}

0&=\int_{\Gamma}\{u_{TY}+(P_x)_Y-kQ_Y\}\phi dYdT\\
\vspace{0.1cm}

&=\int_{\Gamma}\{u_{TY}\phi +[- u^{2}\cos^2 \frac v2 -\frac 12\sin^2 \frac v2+(P-kQ_x)\cos^2\frac v2]\xi \phi\}dYdT\\
\vspace{0.1cm}
&= \int_{\Gamma} \{-u_{Y}\phi_T+[- u^{2}\cos^2 \frac v2 -\frac 12\sin^2 \frac v2+(P-kQ_x)\cos^2\frac v2]\xi \phi\}dYdT\\
\vspace{0.1cm}
 
&= \int_{\Gamma} \{-u_{x}\phi_T\xi\cos^2\frac v2+[- u^{2}\cos^2 \frac v2 -\frac 12\sin^2 \frac v2+(P-kQ_x)\cos^2\frac v2]\xi \phi\}dYdT\\

&= \int_{ \Gamma} \{-u_{x}(\phi_t+u \phi_x)+[ -u^2-\frac 12 u_x^2+ P-kQ_x]  \phi \}dxdt,
\end{array}\]
which implies that the weak form (\ref{weak-form}) is true. Now, we introduce the Radon measures $\{\mu_{(t)}, t\in \mathbb{R} \}$: for any Lebesgue measurable set $\{x\in \mathcal{A}\} $ in $\mathbb{R}$. Supposing the corresponding pre-image set of the transformation is $\{ Y\in \mathcal{G}(\mathcal{A})\}$, we have 
\[\mu_{(t)} (\mathcal{A})=\int_{ \mathcal{G}(\mathcal{A})} \xi \sin^2\frac v2 (t,Y)dY.\]
For every $t\in \mathbb{R} \setminus N$, the absolutely continuous part of $\mu_{(t)}$ w.r.t. Lebesgue measure has density $u_x^2(t,\cdot)$ by (\ref{5.4}). It follows from (\ref{S-1}) that for any test function $\phi(t,x)\in C_c^1(\Gamma)$, 
\[\begin{array}{lll}\vspace{0.1cm}

 -\int_{\mathbb{R}^+}\{\int_{\mathbb{R}} (\phi_t+u \phi_x)d\mu_{(t)} \}dt &=- \int_{\Gamma}\phi_T\xi\sin^2\frac v2dYdT\\
\vspace{0.1cm}

&= \int_{\Gamma}\phi (\xi\sin^2\frac v2)_TdYdT\\
\vspace{0.1cm}

&=  \int_{ \Gamma}2\phi\xi (u^{ 2}-P+kQ_x)\cos \frac v2 \sin \frac v2dYdT\\

&= \int_{\Gamma}2u_x( u^{ 2} - P+kQ_x) \phi dxdt.\\
\end{array}\]
Then, (\ref{weak_en}) is true due to  ${\rm meas}\ (N)=0$.

Finally, let $u_{0,n}$ be a sequence of initial data converging to $u_0$ in $H^1(\mathbb{R}) $. From (\ref{2.1}) and (\ref{S-2}), at time $t=0$ this implies
\[\sup_{Y\in \mathbb{R}}|x_n(0,Y)-x(0,Y)|\rightarrow 0,\ \   \ \  \sup_{Y\in \mathbb{R}}|u_n(0,Y)-u(0,Y)|\rightarrow 0, \]
Meanwhile, $\|v_n(0,\cdot)-v(0,\cdot)\|_{L^2}\rightarrow 0$. This implies that $u_n(T,Y)\rightarrow u(T,Y)$, uniformly for $T,Y$ in bounded sets. Returning to the original coordinates, this yields the convergence 
\[x_n(T,Y)\rightarrow x(T,Y),\ \ \ \  u_n(t,x)\rightarrow u(t,x),\]
uniformly on bounded sets, because all functions $u$, $u_n$ are   H\"{o}lder  continuous. This completes the proof. 
\end{proof}

\renewcommand{\theequation}{\thesection.\arabic{equation}}
\setcounter{equation}{0}
\section{Uniqueness of the Global  Weak Solution}

We shall give the proof  of Theorem \ref{unique}  in  the case $t\geq 0$, and the case $t<0$ can be handled by the similar argument. \\
{\bf Step 1.}  We define
$x(t,\beta)$
to be the unique point $x$ such that
\be\label{6.3}
x(t,\beta)+\mu_{(t)} \{\,(-\infty, x)\,  \}~\leq~ \beta~\leq~x(t,\beta)
+\mu_{(t)}  \{\, (-\infty, x] \}.
\ee 
Recalling  that at every time,  $
\mu_{(t)}$ is absolutely continuous with density $u_x^2$ w.r.t.~Lebesgue measure,
the above definition gives  
\be\label{61}
\beta:=x(t,\beta)+\int_{-\infty}^{x(t,\beta)}u_x^2(t,z)~dz=x(t,\beta)+\mu_{(t)} \{(-\infty,x(t,\beta))\}.
\ee
We study the Lipschitz continuity of $x$ and $u$ as functions
of   $t,\beta$. 
\begin{lemma} \label{6.2} Let $u=u(t,x)$ be the weak  solution of (\ref{E}) satisfying (\ref{weak_en}).
Then, for every $t\ge 0$,  
\begin{itemize}
\item [(i)] $\beta\mapsto x(t,\beta)$ and $\beta\mapsto u(t,\beta):=u(t, x(t, \beta))$
implicitly defined by (\ref{61}) are Lipschitz continuous with the  constant $1$,
\item [(ii)] $t\mapsto x(t,\beta)$
is  Lipschitz continuous with a constant relaying on   $\|u_0\|_{H^{1}}$.
\end{itemize}
\end{lemma}
\begin{proof}
(i) From the definition of $\beta$ in (\ref{61}) , we remark that for any time $ t\geq 0$, $x\mapsto \beta(t,x)$ is right continuous and strictly increasing. Then, its inverse $\beta\mapsto x(t,\beta)$ is well-defined, and is also continuous and nondecreasing.  For any $\beta_1<\beta_2$, we deduce that 
\be\label{La611} 
\begin{array}{lll}
\beta_2-\beta_1&=x(t,\beta_2)-x(t,\beta_1)+\int_{-\infty}^{x(t,\beta_1)}u_x^2(t,z)dz-\int_{-\infty}^{x(t,\beta_2)}u_x^2(t,z)dz\\

&\geq x(t,\beta_2)-x(t,\beta_1)+\mu_{(t)}\{(x(t,\beta_1), x(t,\beta_2)\}.
\end{array}\ee
Hence, we get 
$x(t,\beta_2)-x(t,\beta_1)~\leq~ \beta_2-\beta_1$, and the map  $\beta\mapsto x(t,\beta)$ is   Lipchitz continuous with the constant $1$. For  $\beta\mapsto u(t,\beta)$, from $|u(x_1)-u(x_2)|\leq \int_{x_1}^{x_2}|u_x|dx\leq \frac12 (|x_2-x_1|+\int_{x_1}^{x_2} u_x ^2dx)$  we see that for any $\beta_1<\beta_2$
\[
|u(t,x(t,\beta_2))-u(t,x(t,\beta_1))| \leq  \frac{1}{2}[x(t,\beta_2)-x(t,\beta_1)+\mu_{(t)}\{(x(t,\beta_1),x(t,\beta_2))\}.\]
From (\ref{La611}),   $\beta\mapsto u(t,\beta)$ is Lipchitz continuous with constant $ \frac{1}{2}  <1$. 
 
 (ii) From the Sobolev embedding inequality, we have 
$\|u\|_{L^{\infty}}\leq C\|u\|_{H^{1}}:=C_{\infty}$. 
 Assume $x(t,\beta)=y$. We remark that the family of measures $\mu_{(t)}$ satisfies the balance law (\ref{weak_en}), where for each $t$, the source term $2(u^{ 2}-P+kQ_x)u_x $ in (\ref{weak_en}) has the following estimate by the H\"{o}lder inequality.
\[ \|2(u^{ 2}-P+kQ_x)u_x \|_{L^1} \leq 2\|u_x\|_{L^2} (\|P\|_{L^2}+k\|Q_x\|_{L^2}+\|u^{2}\|_{L^2})\leq C_0, \] 
where $C_0$ depends only on $\|u\|_{H^1}$. Thus, for any $t>\tau$,   from  (\ref{weak_en}), 
 \[\begin{array}{lll}\mu_{(t)}\{(-\infty,y-C_{\infty}(t-\tau))\}&\leq \mu_{(\tau)}\{(-\infty,y)\}+\int_{\tau}^t \|2(u^{ 2}-P+kQ_x)u_x \|_{L^1}dt\\
 
 &\leq \mu_{(\tau)}\{(-\infty,y)\}+C_0(t-\tau).\end{array}\]
 Let $y^{-}(t)=y-(C_{\infty}+C_0)(t-\tau)$. Then, we have
 \[\begin{array}{lll}
 y^{-}(t)+\mu_{(t)}\{(-\infty,y^{-}(t))\}&\leq y-(C_{\infty}+C_0)(t-\tau)+\mu_{(t)}\{(-\infty,y)\}+C_0(t-\tau)\\
 
 &\leq y-\mu_{(\tau)}\{(-\infty,y)\}\leq \beta. 
 \end{array}\]
This implies that $x(t,\beta)\geq y^{-}(t)$ for all $t>\tau$. And we can obtain $x(t,\beta)\leq y^{+}(t):=y+ (C_{\infty}+C_0)(t-\tau)$ by using the similar argument. This completes the proof of the uniformly Lipchitz continuity of the mapping $t\mapsto x(t,\beta)$. 
\end{proof}

 {\bf Step 2.}  The next lemma shows that characteristics can be uniquely determined by an integral equation
combining the characteristic equation and balance law of $u_x^2$, which is crucial to study the uniqueness of the conservative solution of Eq.(\ref{E}).
 
\begin{lemma} \label{6.3} Let $u=u(t,x)$ be the weak  solution of Eq.(\ref{E}) satisfying (\ref{weak_en}). Then, for any $x_0\in \mathbb{R}$ there exists a unique
Lipschitz continuous map  $t\mapsto x(t)$ which
satisfies
\be\label{Characteristic}\frac {d}{dt}x(t)=u (t,x(t)),\ \ \ \ \ \ \  \ x(0)=x_0 \ee 
and \be\label{La621}   \frac{d}{dt}\int_{-\infty}^{x(t)}u_x^2(t,x)  dx = \int_{-\infty}^{x(t)}\left[ 2(u^{ 2}-P+kQ_x)u_x  \right](t,x)dx,
\ x(0)=  x_0,
\ee 
for a.e.~$t\geq 0$.
Furthermore, for any $0\le \tau\leq t$, we have
\be\label{La622}  
u(t, x(t)) - u(\tau, x(\tau)) ~=~-\int_\tau^t (P_x -kQ)(s, x(s))\, ds\,.\ee 
\end{lemma}

\begin{proof}
Firstly,   by the adapted  coordinates $(t, \beta)$, we
write the characteristic
starting at $x_0$ in the form $t\mapsto x(t)=x(t, \beta(t))$, where
$\beta(\cdot)$ is a map to be determined.
Sum  up (\ref{Characteristic}) and (\ref{La621}) and integrate w.r.t.~time. We get
\be\label{La6225}
\beta(t)~= \beta_0+\int_0^t G(s,\beta(s))\, ds\,.\ee
where $\beta=x(t) + \int_{-\infty}^{x(t)}u_x^2(t,x)dx$,  
$\beta_0=x_0 + \int_{-\infty}^{x_0}
(u_{0})_{x}^2(x) dx$ and 
 \be\label{La6224}G(t, \beta)= \int_{-\infty}^{x(s)}[  u_x + 2(u^{ 2}-P+kQ_x)u_x ](s,x)dx.\ee
For each fixed $t\geq 0$,
since the maps $x\mapsto u(t,x)$, $x\mapsto   P(t,x)$ and $x\mapsto   Q_x(t,x)$
are both in $ H^{1}(\mathbb{R})$,  the function $\beta\mapsto G(t,\beta)$
defined at (\ref{La6224}) is
uniformly bounded and absolutely continuous.
Moreover,
\[G_\beta = [   u_x + 2(u^{ 2}-P+kQ_x)u_x ]\, x_\beta = 
\frac{ u_x + 2(u^{ 2}-P+kQ_x)u_x }{1+u_x^2}~ \in ~[-C,\,C]\]
for some constant $C$ depending only on the   $H^{1}$-norm of $u$.
Hence the function $G$ in (\ref{La6224}) is Lipschitz continuous w.r.t.~$\beta$.
We can apply  the   ODE's theory  in the Banach space  of all continuous functions $\beta: \mathbb{R}^+\mapsto \mathbb{R}$ with weighted norm $\|\beta\|:=\sup\limits_{t\geq 0}e^{-2Ct}|\beta(t)|$.  Let $[\Phi \beta](t):= \beta_0+\int_0^t G(\tau,\beta(\tau))d\tau$. Assume $\|\beta-\beta_0\|=\delta>0$. we have $|\beta_1(\tau)-\beta_2(\tau)|\leq \delta e^{2C\tau}$ for all $\tau \geq 0$.  By the Lipchitz continuity of $G$,  
\[\begin{array}{lll}|[\Phi \beta_1](t)-[\Phi \beta_2](t)| \leq C\int_0^t |\beta_1(\tau)-\beta_2(\tau)|d\tau\leq \frac{\delta}{2}e^{2Ct}.\end{array}\]
 Then, $[\Phi \beta]$ is a strict contraction.  (\ref{La6224}) has a unique solution
$t\mapsto \beta(t)$, and the corresponding function $t\mapsto x(t,\beta(t))$  satisfies  (\ref{Characteristic}) and (\ref{La621}).

Secondly, from  the integral equation (\ref{La6225}), we can determine a Lipchitz continuous characteristic $x(t)$ of (\ref{Characteristic}). By the previous construction, the map
$t\mapsto x(t):= x(t, \beta(t))$ provides the unique solution to (\ref{La6225}).
Being the composition of two Lipschitz  functions, the map $t\mapsto x(t,\beta(t))$
is also Lipschitz continuous. To prove that it
 satisfies the ODE for the characteristics of (\ref{Characteristic}), it suffices to show that
 (\ref{Characteristic}) holds at each time $\tau>0$ such that
  \begin{itemize}
\item  [(i)] $x(\cdot)$ is differentiable at $t=\tau$,
\item[(ii)]   the measure $\mu_{(\tau)}$ is absolutely continuous.
\end{itemize}
Assume, on the contrary, that (i) and (ii) hold but $\frac{d}{dt} x(\tau)\not= u (\tau, x(\tau))$.
Let
 \[\label{ass}\frac{d}{dt} x(\tau)~=~ u (\tau, x(\tau))+2\ve_0\] 
 for some $\ve_0>0$(the case $\ve_0 <0$ can be handled similiarly).
To derive a contradiction we
see that, for all $t\in (\tau, \tau+\delta]$,
with $\delta>0$ small enough 
\be\label{xpmt}x^+(t):=
x(\tau) +(t-\tau) [u(\tau, x(\tau))+\ve_0 ]~<~x(t)\,.\ee
We also observe that if $\phi$  is
Lipschitz continuous with compact support then (\ref{weak_en}) is still true. For 
  any $\epsilon>0$ small, we will use the test functions.
\[ \rho^{\epsilon}(s,y):=\left\{\begin{array}{lll}&0 \qquad & {\rm if}  \quad y
\leq -\epsilon^{-1},\\

&(y+\epsilon^{-1}) \qquad &{\rm if} \quad -\epsilon^{-1}\leq y\leq 1-\epsilon^{-1},\\

&1\qquad  & {\rm if} \quad 1-\epsilon^{-1} \leq y\leq x^+(s),\\

&1-\epsilon^{-1}(y-x(s))\qquad &{\rm if} \quad  x^+(s)\leq y\leq x^+(s)+\epsilon,\\

&0 \qquad &{\rm if} \quad y\geq x^+(s)+\epsilon, \end{array}\right.\]
\be\label{timtest}\chi^\epsilon(s):=\left\{\begin{array}{lll}& 0\qquad
&{\rm if} \quad s\leq \tau-\epsilon,\\

& \epsilon^{-1}(s-\tau+\epsilon)\qquad &{\rm if} \quad \tau-\epsilon\leq s\leq \tau,\\

&1\qquad &{\rm if} \quad \tau\leq s\leq  t,\\

& 1-\epsilon^{-1}(s-t) \qquad &{\rm if} \quad t\leq s<t+\epsilon,\\

&0 \qquad &{\rm if} \quad s\geq t+\epsilon. \end{array}\right.\ee 
Let  $\varphi^{\epsilon}(s,y):=\min\{ \rho^{\epsilon}(s,y),
\,\chi^\epsilon(s)\}$. 
Use $\varphi^{\epsilon}$ as test function  in (\ref{weak_en}). 
\be\label{vpe}
\int_{\Gamma}   [u_x^2\varphi^{\epsilon}_t+u u_x^2\varphi^{\epsilon}_x
+  2(u^{ 2}-P+kQ_x)u_x \varphi^{\epsilon}
 ] \, dx dt~=~0.
\ee
For $s\in [\tau+\epsilon, \, t-\epsilon]$, we get  $\varphi^{\epsilon}_x\leq 0$ and $u (s,x)<u (\tau, x(\tau))+\epsilon_0 $ by the H\"older continuity of $u$. Then,
$\varphi^{\epsilon}_t + u (s,x) \varphi^{\epsilon}_x\geq  \varphi^{\epsilon}_t + [u (\tau, x(\tau))+\epsilon_0 ] \varphi^{\epsilon}_x=0$. 
Thus, $\label{bdt}\lim\limits_{\epsilon\to 0}\int_\tau^t(
\int_{x^+(s)-\epsilon}^{x^+(s)+\epsilon}
u_x^2(\varphi^{\epsilon}_t+u \varphi^{\epsilon}_x)(s,x)\, dx) ds
 \geq 0$ as  $t$ is sufficiently close to $\tau$. Since the family of measure $\mu_{(t)}$ depends continuously on $t$
in the topology of weak convergence, by taking  $\epsilon\to 0$ in (\ref{vpe}), we have $\tau,t\notin N$
\be\label{55}\begin{array}{lll}
& \int_{-\infty}^{x^+(t)}u_x^2(t,x)\, dx\\

=&
 \int_{-\infty}^{x(\tau)}u_x^2(\tau,x)\, dx
+ \int_\tau^t\int_{-\infty}^{x^+(s)}2(u^{ 2}-P+kQ_x)u_x \,dxds\\
  &\qquad\qquad  \qquad\qquad \qquad +\lim\limits_{\epsilon\to 0}\int_\tau^t
\int_{x^+(s)-\epsilon}^{x^+(s)+\epsilon}  u_x^2(\varphi^\epsilon_t+u \varphi^\epsilon_x)\, dx ds\\
 
\geq &  \int_{-\infty}^{x(\tau)}u_x^2(\tau,x)\, dx
+ \int_\tau^t\int_{-\infty}^{x(s)} 2(u^{ 2}-P+kQ_x)u_x \,dxds+ o_1(t-\tau),
\end{array}
\ee
where ${o_1(t-\tau)\over t-\tau}\to 0$  is a higher order infinitesimal satisfying 
\[\begin{array}{lll}|o_1(t-\tau)|&=~ |\int_\tau^t \int_{x^+(s)}^{x(s)} 2(u^{ 2}-P+kQ_x)u_x ) \,dx ds |\\

 &\leq~\| 2(u^{ 2}-P+kQ_x) \|_{L^\infty} \int_\tau^t |x(s)-x^+(s)|^{\frac 12}\, \|u_x(s,\cdot)\|_{L^2}\,ds
\\

& \leq~C  (t-\tau)^{\frac{3}{2}}  \rightarrow 0\ \  \ \ \rm{as} \ \ \ t\rightarrow \tau.\end{array}\]
For every $t>\tau$ with $t\notin N$, when $t$ is sufficiently close to $\tau$, by injecting \eqref{xpmt} and \eqref{55} into $\beta(t)$, we deduce that 
\be\label{3.19}
\begin{array}{lll}
 \beta(t)&>   x(\tau)+(t-\tau)[u (\tau,x(\tau))+\ve_0 ]+\int_{-\infty}^{x^+(t)}u_x^2(t,x)dx\\

&  \geq  x(\tau)+(t-\tau)[u (\tau,x(\tau))+\ve_0 ]+\int_{-\infty}^{x(\tau)}u_x^2(\tau,x)dx\\

& \qquad   + \int_\tau^t\int_{-\infty}^{x(s)} 2(u^{ 2}-P+kQ_x)u_x \,dxds+ o_1(t-\tau).
\end{array}
\ee 
On the other hand, from  (\ref{La6225}) and (\ref{La6224}),  a linear approximation yields
\be\label{3.18}\beta(t) =\beta(\tau) +(t-\tau) [ u (\tau, x(\tau)) 
 +\int_{-\infty}^{x(\tau)}  2(u^{ 2}-P+kQ_x)u_x dx]  + o_2(t-\tau)\,,\ee 
with $o_2(t-\tau):=\int_{\tau}^t  \int_{x^+(s)}^{x(s)}[   u_x + 2(u^{ 2}-P+kQ_x)u_x   ] dx ds $, and $\lim\limits_{t\rightarrow \tau} \frac{o_2(t-\tau)}{t-\tau}=0$. 
By combining (\ref{3.18}) and (\ref{3.19}), we see that 
\[\begin{array}{lll}
& x(\tau)+(t-\tau)[u (\tau,x(\tau))+\ve_0 ]+
\int_{-\infty}^{x(\tau)}u_x^2(\tau,x)dx\\

& \qquad   +\int_\tau^t\int_{-\infty}^{x(s)}2(u^{ 2}-P+kQ_x)u_x \, dxds+ o_1(t-\tau)\\

\leq& \beta(\tau) +(t-\tau)  [ u (\tau, x(\tau)) 
 +\int_{-\infty}^{x(\tau)} 2(u^{ 2}-P+kQ_x)u_x \, dx ]+ o_2(t-\tau).\end{array}\] 
Subtracting common terms and dividing both sides by $t-\tau$ and
letting $t\to \tau$, we get $\ve_0\leq 0$, which is  a contradiction.  Then, (\ref{Characteristic}) must hold.

Thirdly, 
  we prove (\ref{La622}).  In  (\ref{weak-form}), let $\phi=\varphi_x$ and $\varphi\in C^\infty_c$. 
Since the map $x\mapsto u(t,x)$ is absolutely continuous, we can
integrate by parts w.r.t.~$x$.
\be\label{CHW2}
  \int_{\Gamma} [u_x\varphi_t +u u_x\varphi_x + (P_x  -kQ)\varphi_x ]\, dxdt
+ \int_{\mathbb{R}} (u_{0})_{x}(x)\varphi(0,x)\, dx~=~0.\ee 
By an
approximation argument,  (\ref{CHW2}) remains
valid for any test function $\varphi$ which is
Lipschitz continuous with compact support.
For any $\epsilon>0$ sufficiently small, we thus consider the function:
\[\varrho^{\epsilon}(s,y):=\left\{\begin{array}{lll}&0 \qquad & {\rm if}\quad y\leq -
\epsilon^{-1},\\

&y+\epsilon^{-1} \qquad & {\rm if}\quad -\epsilon^{-1}\leq y\leq 1-\epsilon^{-1},\\

&1\qquad & {\rm if}\quad 1-\epsilon^{-1} \leq y\leq x(s),\\
 
&1-\epsilon^{-1}(y-x(s))\qquad & {\rm if} \quad  x(s)\leq y\leq x(s)+\epsilon,\\

&0\qquad & {\rm if}\quad y\geq x(s)+\epsilon.\end{array}\right.\]
Then, we define 
$\psi^{\epsilon}(s,y):=\min\{ \varrho^{\epsilon}(s,y),\,\chi^\epsilon(s)\}$, 
where  $\chi^\epsilon(s)$ as in (\ref{timtest}).
Take
$\varphi=\psi^{\epsilon}$  in (\ref{CHW2}) and let
$\epsilon\to 0$. From the continuity of $(P_x -k Q)$,
\be\label{Inteq}\begin{array}{lll}
  \int_{-\infty}^{x(t)} u_x(t,x)\, dx&=  \int_{-\infty}^{x(\tau)}u_x(\tau,x)\, dx
- \int_\tau^t (P_x-k Q)(s, x(s))\,ds\\

& \qquad +\lim\limits_{\epsilon\to 0}\int_{\tau-\epsilon}^{t+\epsilon}\int_{x(s)}
^{x(s)+\epsilon} u_x(\psi^{\epsilon}_t+u \psi^{\epsilon}_x) dx ds\,.\\
\end{array}\ee 
 For every time $s\in [\tau-\epsilon,t+\epsilon]$ by construction, we see that 
\[\psi^{\epsilon}_x(s,y)~=~\epsilon^{-1},
  \psi^{\epsilon}_t(s,y)+u (s, x(s))\psi^{\epsilon}_x(s,y)~=~0
\quad \hbox{for }~ x(s)<y<x(s)+\epsilon.\]
This implies
\[\begin{array}{lll}  \int_{x(s)}^{x(s)+\epsilon}
|\psi^{\epsilon}_t(s,y)+u (s,y)\psi^{\epsilon}_x(s,y)|^{2}dy =\frac{1}{\epsilon^2} \int_{x(s)}^{x(s)+\epsilon}
|u  (s, x(s))-u (s,y) |^{2}dy\\

\leq  \frac{1}{\epsilon } 
 (\max\limits_{x(s)\leq y\leq x(s)+\epsilon}|u (s,y)- u (s, x(s))| )^{2}

\leq  \frac{1}{\epsilon}
 (\int_{x(s)}^{x(s)+\epsilon} | u_x(s,y)| dy )^{2}\\

\leq  \frac{1}{\epsilon} (\epsilon \|u_x(s)\|_{L^2})^2  \rightarrow0 \  \ {\rm as}\ \ \epsilon\rightarrow 0.
\end{array}\] 
Then, we have 
\be\label{11}\begin{array}{lll}|\int_{\tau-\epsilon}^{t+\epsilon}\int_{x(s)}
^{x(s)+\epsilon} u_x(\psi^{\epsilon}_t+u \psi^{\epsilon}_x) dx ds|&\leq C \|u_x\|_{L^2}(\int_{x(s)}^{x(s)+\epsilon}(\psi^{\epsilon}_t+u \psi^{\epsilon}_x)^2dx)^{\frac12}\\

&\leq C \epsilon^{\frac 12} \|u(s)\|_{H^1}^2\rightarrow0 \  \ {\rm as}\ \ \epsilon\rightarrow 0.\end{array}\ee
It follows from (\ref{Inteq}) and  (\ref{11}) that (\ref{La622}) is true.

Finally, we prove the uniqueness of $x(t)$. Assume that there exist two different $x_1(t)$ and $x_2(t)$, which satisfy (\ref{Characteristic}) and (\ref{La621}). Now, choosing the measurable functions $\beta_1$ and $\beta_2$ such that $x_1(t)=x(t,\beta_1(t))$ and $x_2(t)=x(t,\beta_2(t))$. Then, $\beta_1(\cdot)$ and $\beta_2(\cdot)$ satisfy (\ref{La6225}) with the same initial data $x(0)=x_0$. This contradicts with the uniqueness of $\beta$. 
\end{proof}

{\bf Step 3.}  We give some additional properties of   $\beta$ and $u$, as follows.
\begin{lemma} \label{6.4}
 If $u=u(t,x)$ is the weak solution of Eq. (\ref{E}) satisfying (\ref{weak_en}). Then,
\begin{itemize}
\item [(i)] the mapping $(t,\beta)\mapsto u(t,\beta):=u(t,x(t,\beta))$ is Lipchitz continuous with a constant depending only on the norm $\|u_0\|_{H^1}$,
\item [(ii)]   denote $t\mapsto \beta(t;\tau,\beta_0)$ be the solution to the integral equation 
\be\label{La630}\beta(t)~= \beta_0+\int_{\tau}^t G(\tau,\beta(\tau))\, d\tau.\ee We deduce that there exists a constant $C$ such that for any $\beta_{1,0}$,   $\beta_{2,0}$ and any $t,\tau\geq 0$ the corresponding solutions satisfy 
\be\label{La631}|\beta(t;\tau,\beta_{1,0})-\beta(t;\tau,\beta_{2,0})|\leq e^{C|t-\tau|}|\beta_{1,0}-\beta_{2,0}|.\ee
\end{itemize}
\end{lemma}
\begin{proof}
(i) It follows from  (\ref{6.3}), (\ref{La622}), and  (\ref{La6225}) that 
\[\begin{array}{lll}
&| u(t, x(t, \beta_0))- u(\tau, \beta_0) |\\
  \leq&
 | u(t, x(t, \beta_0))- u(t,x(t,\beta(t)))| + | u(t, x(t,\beta(t)))- u(\tau,
x(\tau,\beta(\tau))) | \\

\leq&\frac{1}{2}|\beta(t)-\beta_0| +   (t-\tau)\|P_x-kQ\|_{L^\infty}\\

\leq&C(t-\tau),
\end{array}\]
 where $C:=\frac 12 \|G\|_{L^\infty}+ \|P_x\|_{L^\infty} +k\|Q\|_{L^\infty}>0$   depends only on  $\|u_0\|_{H^1}$.

(ii) It follows from the Lipchitz continuity of $G$ that 
\[|\beta(t;\tau,\beta_{1,0})-\beta(t;\tau,\beta_{2,0})| \leq  |\beta_{1,0}-\beta_{2,0}|+C \int_{\tau}^t|\beta(s;\tau,\beta_{1,0})-\beta(s;\tau,\beta_{2,0})|ds.\]
Then, we can obtain (\ref{La631}) by the Gronwall inequality. 
\end{proof}

{\bf Step 4.}  We give the proof of Theorem \ref{unique}. 
\begin{proof}
 Firstly, by Lemma \ref{6.2} and Lemma \ref{6.4},  the map $(t,\beta)\mapsto (x,u)(t,\beta)$ is
Lipschitz continuous.   An entirely similar argument shows that the maps $\beta\mapsto G(t,\beta):=
G(t,x(t,\beta))$ and $ \beta\mapsto {P } (t,\beta):= {P } (t, x(t,\beta))$,  $\beta\mapsto {P }_x(t,\beta):= {P_x} (t, x(t,\beta))$,
 $ \beta\mapsto Q(t,\beta):= {Q} (t, x(t,\beta))$ and $\beta\mapsto Q_x(t,\beta):= {Q}_x (t, x(t,\beta))$ are also Lipschitz continuous.
By Rademacher's theorem in \cite{Evans2010}, the partial derivatives
$x_t$,  $x_\beta$, $u_t$, $u_\beta$, $G_\beta$, $ {P}_\beta$,  $ {Q}_\beta$, $ ({P}_{x})_\beta$ and $ ({Q}_{x})_\beta$ exist almost everywhere.
Moreover, a.e.~point
$(t,\beta)$ is a Lebesgue point for these derivatives.
Calling $t\mapsto \beta(t, \beta_0)$ the unique solution to
the integral equation (\ref{La6225}), by Lemma \ref{6.3} for a.e.~$\beta_0$  the following holds.
\begin{itemize}
 \item[{\bf (GC)}] For a.e.~$t> 0$, the point
$(t, \beta(t, \beta_0))$ is a Lebesgue point for the partial derivatives
$x_t,  x_\beta, u_t, u_\beta$, $G_\beta$, $ {P}_\beta$,  $ {Q}_\beta$, $ ({P}_{x})_\beta$, $ ({Q}_{x})_\beta$.  Moreover,
$x_\beta(t, \beta(t, \beta_0))>0$ for a.e.~$t> 0$.
\end{itemize}
If (GC) holds, then we say that $t\mapsto
\beta(t,\beta_0)$ is a {\bf good characteristic}. We seek an ODE describing how the quantities
$u_\beta$ and $x_\beta$
vary along a good characteristic.
As in Lemma 6.3, we denote by $t\mapsto Z(t):=\beta(t;\tau, \beta_0)$ to the solution of
(\ref{La630}).  If $\tau,t\notin N$, assuming that $Z(t)$
is a good characteristic and differentiating (\ref{La630}) w.r.t.~$  \beta_0$, we find
\be\label{bode}
{\partial\over\partial\beta_0} Z(t)~=~
1+\int_\tau^t G_\beta(s,Z(s))
\cdot{\partial\over\partial\beta_0} Z(s)\,ds
\ee 
Next, differentiate  $x(t, Z(t))~=~x(\tau, \beta_0)+\int_\tau^t
u (s,\,x(s,Z(s)))\, ds$ w.r.t.~$\beta_0$.
\be\label{xbode}
x_\beta(t, Z(t))\cdot
{\partial\over\partial\beta_0} Z(t)~=~x_\beta(\tau,\beta_0)
+\int_\tau^t u_\beta(s,Z(s))\cdot
{\partial\over\partial\beta_0} Z(s)\,ds.
\ee 
By differentiating w.r.t.~$\beta_0$ the identity (\ref{La622}), we obtain
\be\label{ubode} 
  u_\beta(t, Z(t))\cdot
{\partial\over\partial\beta_0} Z(t)  =  u_\beta(\tau,\beta_0)
-\int_\tau^t (P_x-k Q)_\beta(s,Z(s))\cdot
{\partial\over\partial\beta_0} Z(s)\,ds.\ee
Combining (\ref{bode})--(\ref{ubode}),  we thus obtain the system of ODEs
\be\label{ODES}\left\{\begin{array}{lll} 
&{d\over dt} [{\partial\over\partial\beta_0} Z(t)
 ] = G_\beta(t,Z(t))
\cdot{\partial\over\partial\beta_0} Z(t),\\

 & {d\over dt} [x_\beta(t,Z(t))\cdot
{\partial\over\partial\beta_0} Z(t)
 ]  =  u_\beta(t,Z(t))\cdot
{\partial\over\partial\beta_0} Z(t),\\

 &{d\over dt} [u_\beta(t, Z(t))\cdot
{\partial\over\partial\beta_0} Z(t)
 ]  = -(P_x -kQ)_{\beta}(t,Z(t))\cdot
{\partial\over\partial\beta_0} Z(t).\end{array}\right.\ee 
In particular, the quantities within square brackets on the left hand sides
of (\ref{ODES}) are absolutely continuous.
Recall the fact $P_{xx}=P-(\frac{1}{2}u_x^2+ u^{2})$ and $u_x^2=\frac{1-x_\beta}{x_\beta}$ . 
 From (\ref{ODES}),   along a good characteristic we obtain 
\be\label{xubt}\left\{\begin{array}{lll} 
{d\over dt}
x_\beta+G_\beta x_\beta& = u_\beta\,,\\

 {d\over dt}u_\beta+G_\beta u_\beta&   =  (u^{ 2}-P+kQ_x ) \,{x_\beta}+\frac 12 u_x^2\frac{1}{  1+u_x^2 } ,
\end{array}\right.\ee 
where the first equation is obtained by the first two equations in $(\ref{ODES})$ and
 the second equation is obtained by the first  and third equations in $(\ref{ODES})$

Secondly,  we   go back to the original
$(t,x)$ coordinates and derive an evolution equation
for the partial derivative
$u_x$ along a "good" characteristic curve. Fix a point $(\tau,x_0)$ with $\tau\not\in N$.
Assume that $x_0$ is a Lebesgue point for the map  $x\mapsto u_x(\tau, x)$.
Let  $\beta_0$ be such that $ x_0 = x(\tau, \beta_0)$
and assume that $t\mapsto Z(t)$ is a {\em good characteristic},
so that (GC) holds. We observe that
$u_x^2(\tau,x)~=~{1\over x_\beta(\tau,\beta_0)}-1~\geq~0$, which implies  $x_\beta(\tau,\beta_0)~>~0$. 
As long as $x_\beta>0$, along the characteristic through $(\tau,x_0)$ the
partial derivative $u_x$ can be computed as
$u_x (t,x(t,Z(t)) )~=~{u_\beta(t,Z(t))\over
x_\beta(t,Z(t))}$. 
Using   (\ref{xbode})-(\ref{ubode})
describing the evolution of $u_\beta$ and $x_\beta$, we conclude that the map
$t\mapsto u_x(t, x(t,Z(t)))$ is absolutely continuous
(as long as $x_\beta\not=0$) and satisfies
\[  {d\over dt} u_x(t, x(t,Z(t)))=  \frac{d}{dt}  ({u_\beta(t,Z(t))\over
x_\beta(t,Z(t))})= u^{ 2}-P+kQ_x  -\frac{1}{2} \frac{u_\beta^2}{x_\beta^2}.\]
We remark $u_\beta^2=x_\beta^2u_x^2=x_\beta ( 1-x_\beta ) $,   and as long as $x_\beta>0$ this implies
\be\label{arctan}{d\over dt}\arctan u_x(t, x(t,Z(t)))
= (u^{ 2}-P+kQ_x)x_\beta-\frac{1}{2}(1-x_\beta).\ee
Define the function
\[
v:=\left\{
\begin{array}{lll}
2\arctan u_x \qquad  &\hbox{if}\quad 0<x_\beta\leq 1,\\
\pi \qquad  &\hbox{if}\quad x_\beta=0.
\end{array}\right.
\]
Then, we see that 
\be\label{a22} x_\beta  = {1\over 1+u_x^2} = \cos^2{v\over 2}, \qquad
 1-x_\beta = {u_x^2
\over 1+u_x^2} = \sin^2{v\over 2}\,.\ee 
In the following,  $v$ will be regarded
as a map taking values in the unit circle $\mathcal{S}:= [-\pi,\pi]$
with endpoints identified.
We claim that, along each good characteristic, the map $t\mapsto v(t):=
v(t, x(t,\beta(t;\tau,\beta_0)))$ is absolutely continuous and satisfies
\be\label{Vt}
{d\over dt}v(t) =2 (u^{ 2}-P+kQ_x)\cos^2{v\over 2}-   \sin^2{v\over 2}.\ee 
Indeed, denote by $x_\beta(t)$, $u_\beta(t)$ and $u_x(t)= u_\beta(t)/x_\beta(t)$
the
values of $x_\beta$, $u_\beta$, and $u_x$ along this particular characteristic.
By (GC) we have $x_\beta(t)>0$ for a.e.~$t>0$.

If $\tau$ is any time where $x_\beta(\tau)>0$,
we can find a neighborhood
$I=[\tau-\delta, \tau+\delta]$ such that  $x_\beta(t)>0$ on $I$.
By (\ref{arctan}) and (\ref{a22}), $v= 2\arctan(u_\beta/x_\beta)$
is absolutely continuous  restricted to $I$ and satisfies (\ref{Vt}).
To prove our claim, it thus remains to show that $t\mapsto v(t)$
is continuous on the null set $N$ of times where $x_\beta(t)=0$.
Suppose $x_\beta(0)=0$. From the fact that the identity
$u_x^2(t)~=~{{1-x_\beta(t)}\over{x_\beta(t)}}$ 
valids as long as $x_\beta>0$, it is clear that $u_x^2\to +\infty$ as $t\to  0$
and $x_\beta(t)\to 0$.   This implies $v(t)=2\arctan u_x(t)
\to \pm\pi$.  Since we identify the points $\pm\pi$, we get the
continuity of $v$ for all $t\geq 0$, proving our claim.

Thirdly,  let   $u=u(t,x)$ be a global weak solution of (\ref{E}) satisfying (\ref{weak_en}).
As shown by the previous analysis, in terms of the variables $t,\beta$ the quantities
$x,u,v$ satisfy the semil-inear system
\be\label{xuv}\left\{\begin{array}{lll}
{d\over dt} \beta(t, \beta_0)&= G(t, \beta(t, \beta_0)),\\

  {d\over dt} x(t,\beta(t,\beta_0))&= u (t,\beta(t,\beta_0)),\\
  
  {d\over dt} u(t,\beta(t,\beta_0))&=-P_x+kQ,\\

 {d\over dt} v(t,\beta(t,\beta_0))&=
2 (u^{ 2}-P+kQ_x)\cos^2{v\over 2}-  \sin^2{v\over 2}\,.
\end{array}\right.\ee 
We recall that $P$, $Q$ and $G$ were defined at (\ref{PQ}) and (\ref{La6224}), respectively.
The function $P$, $P_x$, $Q$ and $Q_x$ admit  the representations in terms of   $\beta$, 
\[P(x(\beta))=\frac{1}{2} \int_{-\infty}^{+\infty}e^{-|\int_{\beta}^{\beta^{'}}\cos^2 \frac{v(s)}{2} ds|}[ u^{2}\cos^2\frac{v(\beta^{'})}{2} +\frac{1}{2}\sin^{2}\frac{v(\beta^{'})}{2}]   d \beta^{'},\]
\[ P_x(x(\beta)) =\frac 12(\int_{\beta}^{+\infty}-\int_{-\infty}^{\beta})e^{-|\int_{\beta}^{\beta^{'}}\cos^2\frac{v(s)}{2}  ds|}  [ u^{2}\cos^2\frac{v(\beta^{'})}{2} +\frac{1}{2}\sin^{2}\frac{v(\beta^{'})}{2} ]  d\beta^{'},\]
 \[Q(x(\beta))=\frac 12\int_{-\infty}^{+\infty}e^{-|\int_{\beta}^{\beta^{'}}\cos^2\frac{v(s)}{2}  ds|} u\cos^2 \frac  {v(\beta^{'})}{2} d\beta^{'},\]
\[ Q_x(x(\beta))= \frac{1 }{2} (\int_{\beta}^{+\infty}-\int_{-\infty}^{\beta})e^{-|\int_{\beta}^{\beta^{'}}\cos^2\frac{v(s)}{2}  ds|}   u\cos^2 \frac  {v(\beta^{'})}{2}  d\beta^{'}.\]
For every $\beta_0\in \mathbb{R}$, the initial data is in the following
\be\label{ico}\left\{\begin{array}{lll}
  \beta(0, \beta_0)&=  \beta_0 ,\\

 x(0,\beta_0)&=  x(0,\beta_0),\\
  
  u(0, \beta_0)&=u_0(x(0, \beta_0)),\\

  v(0,\beta_0)&=2\arctan (u_0)_x(x(0,\beta_0)).
\end{array}\right.\ee 
 By the Lipschitz continuity of all coefficients,
the Cauchy problem (\ref{xuv})-(\ref{ico}) has a unique
solution, globally defined for all $t\geq 0$, $x\in  \mathbb{R}$.
Finally, to complete the proof of uniqueness,   we consider two solutions
$u,\tilde u$ of Eq.(\ref{E})
with the same initial data $u_0\in H^{1}(\mathbb{R})$.
For a.e.~$t\geq 0$  the
corresponding  Lipschitz continuous
maps $\beta\mapsto x(t,\beta)$, $\beta\mapsto \tilde x(t,\beta)$
are strictly increasing.  Then,  they have continuous inverses, saying
$x\mapsto \beta^* (t,x)$, $x\mapsto \tilde \beta^* (t,x)$.
By the previous analysis, the map    $(t,\beta)\mapsto (x,u,v)(t,\beta)$
is uniquely determined by the initial data $u_0$.  Therefore
$x(t,\beta)~=~\tilde x(t,\beta)$ and  $u(t,\beta)=\tilde u(t,\beta)$.
In turn, for a.e.~$t\geq 0$ this implies
\[u(t,x) ~=~u(t,\beta^*(t,x))~=~\tilde u(t,\tilde \beta^*(t,x))~=~\tilde u(t,x).\]
\end{proof}

{\bf Acknowledgment. }

This paper was   done when the author visited the School of Mathematics of the Georgia Institute of Technology. The author would like to thank the hospitality of the School of Mathematics. The author would like to thank the referees for the helpful suggestions.
This paper  was  supported by the National Natural Science Foundation of China grant No. 11371267, No. 11501395,  and  Excellent Youth Foundation of Sichuan  Scientific Committee grant No. 2014JQ0039  in China.

\end{document}